\documentclass[10pt]{amsart}
\usepackage{tikz-cd}
\usepackage{amssymb}
\usepackage{mathrsfs}
\usepackage[shortlabels]{enumitem}
\usepackage[colorlinks,linkcolor=black,citecolor=black,urlcolor=black]{hyperref}
\usepackage[color = blue!20, bordercolor = black, textsize = tiny]{todonotes}
\usepackage[capitalise]{cleveref}
\usepackage{comment}

\numberwithin{equation}{section}
\newtheorem{thm}[subsection]{Theorem}

\newtheorem{lem}[subsection]{Lemma}
\newtheorem{prop}[subsection]{Proposition}
\newtheorem{thmalpha}{Theorem}

\theoremstyle{definition}
\newtheorem{df}[subsection]{Definition}
\newtheorem{rmk}[subsection]{Remark}
\newtheorem{exm}[subsection]{Example}
\newtheorem{const}[subsection]{Construction}

\newcommand{\bE}{\mathbf{E}}

\newcommand{\A}{\mathbb{A}}

\newcommand{\N}{\mathbb{N}}
\renewcommand{\P}{\mathbb{P}}

\newcommand{\Z}{\mathbb{Z}}

\newcommand{\cC}{\mathcal{C}}

\newcommand{\cE}{\mathcal{E}}
\newcommand{\cF}{\mathcal{F}}
\newcommand{\cG}{\mathcal{G}}

\newcommand{\cI}{\mathcal{I}}
\newcommand{\cJ}{\mathcal{J}}

\newcommand{\cO}{\mathcal{O}}

\newcommand{\cV}{\mathcal{V}}

\newcommand{\rD}{\mathrm{D}}

\newcommand{\rN}{\mathrm{N}}

\newcommand{\sT}{\mathscr{T}}

\DeclareMathOperator{\Hom}{Hom}
\DeclareMathOperator{\Spec}{Spec}

\newcommand{\colim}{\mathop{\mathrm{colim}}}

\newcommand{\id}{\mathrm{id}}
\newcommand{\ul}{\underline}

\newcommand{\lSm}{\mathrm{lSm}}

\newcommand{\lSch}{\mathrm{lSch}}

\newcommand{\Sm}{\mathrm{Sm}}
\newcommand{\SmlSm}{\mathrm{SmlSm}}

\newcommand{\sat}{\mathrm{sat}}
\newcommand{\gp}{\mathrm{gp}}

\newcommand{\divi}{\mathrm{div}}

\newcommand{\Bl}{\mathrm{Bl}}

\newcommand{\Gys}{\mathrm{Gys}}

\newcommand{\gr}{\mathrm{gr}}

\DeclareMathOperator{\cofib}{cofib}

\newcommand{\Sh}{\mathrm{Sh}}
\newcommand{\Sp}{\mathrm{Sp}}
\newcommand{\Spc}{\mathrm{Spc}}

\newcommand{\Mod}{\mathrm{Mod}}

\newcommand{\Sch}{\mathrm{Sch}}
\newcommand{\ad}{\mathrm{ad}}
\newcommand{\Ex}{\mathrm{Ex}}
\newcommand{\PrL}{\mathrm{Pr}^\mathrm{L}}

\newcommand{\Prst}{\mathrm{Pr}^\mathrm{st}}

\newcommand{\CAlg}{\mathrm{CAlg}}

\newcommand{\op}{\mathrm{op}}

\newcommand{\ev}{\mathrm{ev}}

\newcommand{\comp}{\mathrm{comp}}
\newcommand{\Fin}{\mathrm{Fin}}

\newcommand{\logSH}{\mathrm{logSH}}

\newcommand{\SH}{\mathrm{SH}}

\newcommand{\DM}{\mathrm{DM}}

\newcommand{\MS}{\mathrm{MS}}

\newcommand{\BMS}{\mathrm{BMS}}

\newcommand{\crys}{\mathrm{crys}}

\newcommand{\sNis}{\mathrm{sNis}}
\newcommand{\syn}{\mathrm{syn}}

\newcommand{\blogTHH}{\mathrm{log}\mathbf{THH}}
\newcommand{\logTC}{\mathrm{logTC}}

\newcommand{\blogTC}{\mathrm{log}\mathbf{TC}}

\newcommand{\lc}{\mathrm{lc}}

\newcommand{\Th}{\mathrm{Th}}
\newcommand{\unit}{1}

\newcommand{\MZ}{\mathbf{M}\mathbb{Z}}
\newcommand{\KGL}{\mathbf{KGL}}
\newcommand{\logKGL}{\mathrm{log}\mathbf{KGL}}
\newcommand{\bOmega}{\mathbf{\Omega}}

\usepackage[bbgreekl]{mathbbol}
\usepackage{relsize}

\DeclareSymbolFontAlphabet{\mathbb}{AMSb} 
\DeclareSymbolFontAlphabet{\mathbbl}{bbold} 

\newcommand{\bPrism}{\widehat{\mathbf{\Delta}}}

\renewcommand{\cong}{\simeq}

\begin{document}
\title{Poincar\'e duality in logarithmic motivic homotopy theory}
\author{Doosung Park}
\address{Department of Mathematics and Informatics, University of Wuppertal, Germany}
\email{dpark@uni-wuppertal.de}
\date{\today}
\begin{abstract}
By adapting arguments of Annala-Hoyois-Iwasa in the log setting,
we prove Poincar\'e duality for smooth projective morphisms in logarithmic motivic homotopy theory.
As an application,
we show that the crystalline cohomology of a log compactification is independent of the choice.
\end{abstract}
\maketitle

\section{Introduction}

Poincar\'e duality in motivic homotopy theory asserts that for every smooth projective (or proper) morphism of schemes $f\colon X\to S$,
the motive $M_S(X)$ is dualizable with dual
\[
\Th_X^{-1}(\Omega_f^1)
:=
f_\sharp (\Th_X(\Omega_f^1)^{-1}),
\]
where $\Th_X(\Omega_f^1)$ denotes the Thom space associated with the locally free sheaf $\Omega_f^1:=\Omega_{X/S}^1$.
One of the main advantages of this motivic formulation is that it applies uniformly to all cohomology theories representable in the motivic homotopy category.

\

In the setting of $\DM(k)$ \cite{MVW} over a perfect field $k$,
Friedlander-Voevodsky \cite[Theorem 8.2]{MR1764201} proved Poincar\'e duality assuming resolution of singularities.
D\'eglise \cite[Theorem 5.23]{MR2466188} removed the assumption of resolution of singularities.
Beilinson-Vologodsky \cite[6.7]{MR2399083} extended this to smooth proper morphisms.

\

In the setting of $\SH$ \cite{zbMATH01194164},
Voevodsky-R\"ondig-Ayoub \cite[Scholie 1.4.2]{Ayo071} (see also \cite[Theorem 2.4.28]{CD19}) proved Poincar\'e duality in the form
\[
f_\sharp
\xrightarrow{\simeq}
f_* \Sigma^{\Omega_f^1},
\]
where $\Sigma^{\Omega_f^1}:=\otimes_X \Th_X(\Omega_f^1)$.
Cisinski-D\'eglise \cite[Theorem 2.4.50]{CD19} extended this to smooth proper morphisms.

\

Recently, Annala-Iwasa-Hoyois \cite[Theorem 1.1(i)]{AHI} proved Poincar\'e duality for their $\infty$-category of motivic spectra $\MS$ assuming the existence of Tang's Gysin morphisms with naturality as stated in \cite[Theorem 2.3]{AHI},
which is unavailable at the time of writing.
The crucial distinction between $\MS$ and $\SH$ is that $\A^1$ is not contractible in $\MS$.
Consequently, many non $\A^1$-cohomology theories are representable in $\MS$ but not in $\SH$.

\

In contrast to $\MS$,
log Gysin morphisms are available in the logarithmic motivic homotopy $\infty$-category $\logSH$ due to Binda-Park-{\O}stv{\ae}r \cite{logSH}.
Using log Gysin morphisms instead,
we establish Poincar\'e duality for $\logSH$ as follows.
Hence Poincar\'e duality applies to non $\A^1$-invariant cohomology theories that admit suitable logarithmic extensions,
e.g.\ Hodge cohomology, topological cyclic homology, prismatic cohomology, syntomic cohomology, etc.

\begin{thmalpha}[See Theorem \ref{local.2}]
\label{intro.6}
Let $f\colon X\to S$ be a projective morphism of schemes.
Then
\[
f_\sharp
\xrightarrow{\simeq}
f_* \Sigma^{\Omega_f^1}
\colon
\logSH(X)
\to
\logSH(S).
\]
\end{thmalpha}

The global structure of the proof of Theorem \ref{intro.6} is analogous to the proof of \cite[Theorem 1.1(i)]{AHI}, but in several local arguments we proceed differently since the naturality of Tang's Gysin morphisms as stated in \cite[Theorem 2.3]{AHI} is not available for log Gysin morphisms too.
For example, compare \cite[Proposition 4.4]{AHI} and Proposition \ref{Gysin.13}.
Instead, we prove a weaker naturality of log Gysin morphisms in Theorem \ref{Gysin.11} arguing as D\'eglise \cite[Theorem 4.32]{MR2466188},
which is enough for our purposes.

\

We use Theorem \ref{intro.6} to prove that the localization property is false for closed non-open immersions of smooth schemes,
see Proposition \ref{local2.2}.

\

A log smooth fs log scheme can be viewed as a ``smooth scheme with boundary.''
The version of Poincar\'e duality for manifolds with boundary, which is called Lefschetz duality,
asserts
\[
H^{*}(M,\partial M)
\simeq
 H_{n-*}(M)
\]
for every $n$-dimensional orientable compact manifold $M$ with boundary $\partial M$.
Inspired by these,
we establish the following motivic version of Lefschetz duality.

\begin{thmalpha}[See Theorem \ref{local.3}]
Let $Z\to X$ be a closed immersion of smooth projective schemes over a scheme $S$.
Then $\Sigma^\infty X/Z\in \logSH(S)$ is dualizable with dual
\(
\Th_{(X,Z)}^{-1}(\Omega_{(X,Z)/S}^1),
\)
where $(X,Z)$ is the associated log smooth fs log scheme with boundary $Z$.
\end{thmalpha}

This is a crucial ingredient of Theorem \ref{blow.6},
where we prove that if Poincar\'e duality holds for three corners of a smooth blow-up square, then Poincar\'e duality holds for the remaining corner too.

\

Assuming the existence of Tang's Gysin morphisms,
as a consequence of Poincar\'e duality for $\MS$,
Annala-Hoyois-Iwasa \cite[Proposition 6.27]{AHI} established a cohomology theory of smooth schemes over a perfect field $k$ of characteristic $p>0$ that agrees with the crystalline cohomology of any log compactification.
Previously, Ertl-Shiho-Sprang \cite{ESS} established such a cohomology theory assuming resolution of singularities and weak factorization,
and Merici \cite{Mer25} established it assuming resolution of singularities and represented it in $\DM(k)$.
We repeat the argument of Annala-Hoyois-Iwasa to prove the following result.

\begin{thmalpha}[See Theorem \ref{coh.8}]
\label{intro.5}
Let $k$ be a perfect field of characteristic $p>0$.
Then there exists a cdh sheaf of complexes
\(
R\Gamma_\crys^\lc
\)
on the category of schemes over $k$ such that
\[
R\Gamma_\crys^\lc(X-\partial X)
\simeq
R\Gamma_\crys(X)
\]
for every projective log smooth fs log scheme $X$ over $k$.
\end{thmalpha}

For this,
Annala-Hoyois-Iwasa showed that the crystalline cohomology spectrum is a module over the $\A^1$-localized sphere spectrum.
This is directly shown using the $\A^1$-invariance of $K$-theory and the cyclotomic trace.
To achieve this in the log setting,
we need the log cyclotomic trace constructed in \cite{logSHF1} and \cite{logSHF2},
which is a nontrivial input.

\

In summary,
$\logSH$ is a well-behaved $\infty$-category
equipped with log Gysin morphisms and log cyclotomic traces.
These two tools allow us to prove Poincar\'e duality for $\logSH$ and the independence of the crystalline cohomology of log compactifications.

\subsection*{Notation and conventions}
We refer to Ogus' book \cite{Ogu} for log geometry.
Unless otherwise stated,
all schemes and log schemes in this paper are assumed to be quasi-compact and quasi-separated,
and all log schemes are assumed to equip a Zariski log structure.
We employ the following notation throughout the paper.

\

\begin{tabular}{l|l}
$\Sch$ & category of qcqs schemes
\\
$\lSch$ & category of qcqs fs log schemes
\\
$\Sm$ & class of smooth morphisms in $\Sch$
\\
$\lSm$ & class of log smooth morphisms in $\lSch$
\\
$\lSm^\sat$ & class of saturated log smooth morphisms in $\lSch$
\\
$\Spc_*$ & $\infty$-category of pointed spaces
\\
$\Sp$ & $\infty$-category of spectra
\\
$\Hom_{\cC}$ & hom space in $\cC$
\\
$\hom_{\cC}$ & hom spectrum in a stable $\infty$-category $\cC$
\\
$\Sh_t(\cC,\cV)$ & $\infty$-category of $t$-sheaves on $\cC$ with values in $\cV$
\\
$\id \xrightarrow{\ad} GF$ & unit of an adjoint pair $(F,G)$
\\
$FG \xrightarrow{\ad'} \id$ & counit of an adjoint pair $(F,G)$
\end{tabular}

\subsection*{Acknowledgement}
We thank Federico Binda and Hiroyasu Miyazaki for helpful comments on the subject of this paper.
This research was conducted in the framework of the DFG-funded research training group GRK 2240: \emph{Algebro-Geometric Methods in Algebra, Arithmetic and Topology}.

\section{Saturated logarithmic motivic homotopy theory}
\label{sat}

Recall from \cite{zbMATH07027475} that a morphism of fs log schemes $X\to S$ is \emph{saturated} if for every morphism of fs log schemes $S'\to S$, the fiber product of $X$ and $S'$ over $S$ in the category of log schemes is saturated.
Note that the class of saturated morphisms in the category of fs log schemes is closed under compositions and pullbacks.

Let $\lSm^\sat$ denote the class of saturated log smooth morphisms,
let $\divi$ denote the class of dividing covers \cite[Definition 2.2.8]{logSH},
and let $\sNis$ denote the strict Nisnevich topology \cite[Definition 3.1.4]{logDM}.
For $S\in \lSch$,
we define
\[
\logSH^\sat(S)
:=
\Sp_{\P^1}(\Sh_\sNis(\lSm_S^\sat,\Sp)[\square^{-1},\divi^{-1}]),
\]
where $\Sp_{\P^1}$ denote the $\P^1$-stabilization,
and $\square:=(\P^1,\infty)$.
This is a variant of $\logSH(S)$ due to Binda, Park, and {\O}stv{\ae}r \cite[Definition 2.4.13]{logSH}.
If $S\in \Sch$,
then $\logSH^\sat(S)=\logSH(S)$.
If $S\in \lSch$,
then $\logSH^\sat(S)$ has a technical advantage,
see e.g.\ Proposition \ref{Gysin.26}.

In this paper,
for flexibility,
we will often work with a general axiomatic setting instead of $\logSH^\sat$ as follows:
Let
\[
\sT
\in \Sh_\sNis(\lSch,\CAlg(\Prst))
\]
be a strict Nisnevich sheaf of presentably symmetric monoidal stable $\infty$-categories satisfying the following conditions:
\begin{itemize}
\item For every morphism of fs log schemes $f\colon X\to S$,
$f^*:=\sT(f)$ admits a right adjoint $f_*$.
If $f\in \lSm^\sat$,
then $f^*$ admits a left adjoint $f_\sharp$.
\item ($\lSm^\sat$ base change formula)
For every cartesian square of fs log schemes
\[
\begin{tikzcd}
X'\ar[r,"g'"]\ar[d,"f'"']&
X\ar[d,"f"]
\\
S'\ar[r,"g"]&
S
\end{tikzcd}
\]
such that $f\in \lSm^\sat$,
the exchange transformation
\[
\Ex
\colon
f_\sharp'g'^*
\to
g^*f_\sharp
\]
is an isomorphism.
\item ($\lSm^\sat$ projection formula)
For every $f\in \lSm^\sat$,
the exchange transformation
\[
\Ex
\colon
f_\sharp((-)\otimes_X f^*(-))
\to
f_\sharp(-)\otimes_S (-)
\]
is an isomorphism.
\item ($\P^1$-stability) The Tate twist $1(1)$ is invertible in $\sT(S)$.
\item Let $S\in \lSch$.
Then $\sT(S)$ is generated under colimits, shifts, and Tate twists by $M_S(X):=f_\sharp \unit_X$ for $X\in \lSm_S^\sat$,
where $f\colon X\to S$ is the structural morphism.
\item ($\square$-invariance)
Let $S\in \lSch$.
Then the functor $p^*\colon \sT(S)\to \sT(S\times \square)$ is fully faithful,
where $p\colon S\times \square\to S$ is the projection.
\item (div-invariance)
Let $S\in \lSch$.
Then for every dividing cover $f\colon Y\to X$ in $\lSm_S^\sat$,
$M_S(f)\colon M_S(Y)\to M_S(X)$ is an isomorphism.
\end{itemize}

For example, a formal argument shows that $\logSH^\sat$ satisfies these conditions.
Moreover,
the universal property of $\logSH^\sat$ as in \cite[Proposition 3.3.2]{logSH} yields a natural symmetric monoidal colimit preserving functor
\[
\logSH^\sat(S)
\to
\sT(S)
\]
for $S\in \lSch$ sending $\Sigma^\infty X_+$ to $M_S(X)$ for every $X\in \lSm_S^\sat$.
When $Y\to X$ is a morphism in $\lSm_S^\sat$, we set
\[
M_S(X/Y):=\cofib(M_S(Y)\to M_S(X)).
\]
When
\[
\begin{tikzcd}
W\ar[d]\ar[r]&V\ar[d]
\\
Y\ar[r]&X
\end{tikzcd}
\]
is a commutative square in $\lSm_S^\sat$,
we set
\[
M_S\left(\frac{X/Y}{V/W}\right)
:=
\cofib(\colim(M_S(Y)\leftarrow M_S(W) \to M_S(V))\to M_S(X)).
\]

\section{Thom space}
We keep working with $\sT$ in \S \ref{sat}.
In this section,
we generalize the notion of Thom spaces \cite[Definition 3.2.4]{logSH} to the case where the base $S$ is an arbitrary fs log scheme.
We also discuss a technical result on the invariance of Thom spaces in Proposition \ref{blow.5}.

The following result allows us to reduce to the case of strict smooth schemes in many situations.

\begin{prop}
\label{Gysin.26}
Let $Z\to Y\to X$ be strict closed immersions in $\lSm_S^\sat$.
Then Zariski locally on $X$,
$f$ admits a factorization $X\to S'\to S$ with $S'\in \lSm_S^\sat$ such that $X,Y,Z\to S'$ are strict smooth.
\end{prop}
\begin{proof}
Let $x$ be a point of $X$.
We may assume that $S$ admits a chart $P$ neat at the image of $x$ in $S$.
By the implication (1)$\Rightarrow$(4) in \cite[Theorem III.2.5.5]{Ogu} and \cite[Proposition A.4]{divspc},
in a Zariski neighborhood of $x$ in $X$,
$X\to S$ admits a chart $\theta\colon P\to Q$ neat at $x$,
and $X\to S':=S\times_{\A_P} \A_Q$ is strict smooth.

Observe that $\theta$ is saturated.
By \cite[Theorem IV.3.3.1]{Ogu},
$\theta$ is injective,
and by \cite[Theorem I.4.8.14(4)]{Ogu},
the cokernel of $\theta^\gp$ is torsion free.
Hence we have $S'\in \lSm_S^\sat$ since $S'\in \lSm_S$ by \cite[Theorem IV.3.1.8]{Ogu}.

We claim that $Y,Z\to S'$ are strict smooth Zariski locally on $X$.
We only need to deal with $Y$ without loss of generality.
If $x\in Y$,
then since $\theta$ is also a chart of $Y\to S$ neat at $x$,
\cite[Proposition A.4]{divspc} implies that $Y\to S'$ is strict smooth.
If $x\notin Y$,
then consider the neighborhood $X-Y$ of $x$ to conclude.
\end{proof}

For a strict closed immersion $Z\to X$ in $\lSm_S^\sat$ with $S\in \lSch$,
we will use the adding boundary notation
\[
(X,Z)
:=
(\Bl_{\ul{Z}}\ul{X},E) \times_{\Bl_{\ul{Z}}\ul{X}} X,
\]
where $(\Bl_{\ul{Z}}\ul{X},E)$ is the fs log scheme with the underlying scheme $\Bl_{\ul{Z}}\ul{X}$ equipped with the Deligne-Faltings log structure \cite[\S III.1.7]{Ogu} associated with the exceptional divisor $E$.
We claim $(X,Z)\in \lSm_S^\sat$.
For this,
We can work Zariski locally on $X$.
Hence by Proposition \ref{Gysin.26},
we may assume that $X$ and $Z$ are strict smooth over $S$.
Then the claim is obvious.

\begin{df}
Let $S\in \lSch$, and let $X\in \lSm_S^\sat$.
A \emph{vector bundle $\cE\to X$} is a strict morphism of fs log schemes such that $\ul{\cE}\to \ul{X}$ is a vector bundle.
The \emph{Thom space of $\cE\to X$} is defined to be
\begin{equation}
\label{Gysin.0.6}
\Th_X(\cE)
:=
M_S(\cE/(\cE,X)),
\end{equation}
where $X$ in $(\cE,X)$ is regarded as the zero section of $\cE$.
The \emph{$\cE$-twisted suspension} is defined to be
\begin{equation}
\label{Gysin.0.7}
\Sigma^{\cE}
:=
\otimes_X \Th_X(\cE)
\colon
\sT(X)
\to
\sT(X).
\end{equation}
\end{df}

Let $D\colon \sT(S)\to \sT(S)$ denote the functor $\ul{\Hom}_{S}(-,1_S)$,
where $\ul{\Hom}_S$ denotes the internal hom in $\sT(S)$.

\begin{prop}
Let $X\in \lSch$,
and let $\cE\to X$ be a vector bundle.
Then $\Th_X(\cE)$ is invertible in $\sT(X)$.
\end{prop}
\begin{proof}
We may assume that $X$ has the trivial log structure since the general case follows using the pullback functor for the morphism $X\to \ul{X}$ removing the log structure.
We need to show that the induced morphism
\[
\Th_X(\cE)\otimes_X D(\Th_X(\cE))
\to
1_X
\]
is an isomorphism.
We can work Zariski locally on $X$,
so we may assume that $\cE$ is a trivial vector bundle of rank $d$.

In this case,
we have $\Th_X(\cE)\simeq \unit_X(d)[2d]$ by  \cite[Proposition 3.2.7]{logSH},
whence the result follows.
\end{proof}

\begin{df}
Let $S\in \lSch$, and let $X\in \lSm_S^\sat$.
For a vector bundle $\cE\to X$,
we set
\[
\Th_X^{-1}(\cE)
:=
f_\sharp (\Th_X(\cE)^{-1}).
\]
We also set
\[
\Sigma^{-\cE}:= \otimes_X \Th_X(\cE)^{-1} 
\colon
\sT(X)\to\sT(X),
\]
which is quasi-inverse to $\Sigma^\cE$.
\end{df}

The monoidality of $f^*$ yields a natural isomorphism
\[
f^*
\Sigma^{\cE}
\simeq
\Sigma^{f^* \cE} f^*.
\]
The corresponding exchange transformation
\[
\Ex
\colon
\Sigma^{\cE} f_*
\to
f_* \Sigma^{f^*\cE}
\]
is an isomorphism since $\Sigma^{\cE}$ and $\Sigma^{f^*\cE}$ are equivalences of $\infty$-categories.

For a ring $R$ with elements $r_1,\ldots,r_n$,
let $(R,\langle r_1,\ldots,r_n\rangle)$ denote the log ring $\N^n\to R$ that sends $(m_1,\ldots,m_n)\in \N^n$ to $r_1^{m_1}\cdots r_n^{m_n}\in R$.

For $S\in \Sch$,
let $\SmlSm_S$ denote the full subcategory of $\lSm_S$ spanned by those $X$ such that $\ul{X}\in \Sm_S$.
A morphism $f\colon Y\to X$ in $\SmlSm_S$ is called \emph{an admissible blow-up along a smooth center} if $\ul{Y}\to \ul{X}$ is the blow-up along a smooth center and $Y-\partial Y\to X-\partial X$ is an isomorphism,
see \cite[Definition 7.2.3]{logDM} for the equivalent one.

The induced morphism
\[
\Spec(R[x/y,y],\langle x/y,y\rangle)
\cup
\Spec(R[x,y/x],\langle x,y/x\rangle)
\to
\Spec(R[x,y],\langle x,y\rangle)
\]
is a dividing cover.
We will use some pullbacks of this dividing cover in the proof of the following result:

\begin{prop}
\label{blow.5}
Let $S\in \lSch$,
let $D\to X$ be a purely $1$-codimensional strict closed immersion in $\lSm_S^\sat$,
and let $\cE$ be a vector bundle on $(X,D)$.
Then the morphism
\[
\Th_{(X,D)}(\cE(-D))
\to
\Th_{(X,D)}(\cE)
\]
in $\sT(X,D)$ induced by the canonical inclusion $\cE(-D)\to \cE$ is an isomorphism.
\end{prop}
\begin{proof}
Using Proposition \ref{Gysin.26},
we reduce to the case where $S\in \Sch$ and $D,X\in \Sm_S$.
Then $D$ is a smooth divisor on $X$.

We can work Zariski locally on $X$.
Hence we may assume that $X=\Spec(A)$, $D=\mathrm{div}(a)$ with $a\in A$, and $\cE=\cO^n$ for some integer $n\geq 0$.
Then we have
\[
\cE=\Spec(A[t_1,\ldots,t_n],\langle a\rangle),
\;
\cE(-D)=\Spec(A[t_1/a,\ldots,t_n/a],\langle a \rangle).
\]
We set $Y:=(X,D)$ for simplicity of notation.

Step 1.
\emph{Construction of $V$.}
Let $V\to \cE$ be the admissible blow-up along $D$.
Consider
\[
U_i:=\Spec(A[t_i,a/t_i,t_1/t_i,\ldots,t_n/t_i],\langle t_i,a/t_i\rangle)
\]
and $U:=U_1\cup \cdots \cup U_n$.
Then $V=\cE(-D)\cup U$.
Regard $Y$ as the zero sections of $\cE$ and $\cE(-D)$,
and regard $Y$ as a strict closed subscheme of $V$ via the open immersion $\cE(-D)\to V$.
Observe that $Y$ does not intersect with $U$.
So far,
we have an induced commutative diagram
\[
\begin{tikzcd}[row sep=small]
&
Y\ar[d]
\\
U\cap \cE(-D)\ar[d]\ar[r]&
\cE(-D)\ar[d]
\\
U\ar[r]&
V\ar[r]&
\cE.
\end{tikzcd}
\]

Step 2. \emph{Using Zariski descent}.
We have
\begin{gather*}
M_Y(V/U)\simeq M_Y(\cE(-D)/U\cap \cE(-D)),
\\
M_Y((V,Y)/U)\simeq M_Y((\cE(-D),Y)/U\cap \cE(-D))
\end{gather*}
by Zariski descent.
Hence
\begin{align*}
& M_Y(V/(V,Y))
\simeq 
M_Y\left(\frac{V/U}{(V,Y)/U}\right)
\\
\simeq &
M_Y\left(\frac{\cE(-D)/U\cap \cE(-D)}{(\cE(-D),Y)/U\cap \cE(-D)}\right)
\simeq
M_Y(\cE(-D)/(\cE(-D),Y)).
\end{align*}
It remains to show $M_Y(V)\simeq M_Y(\cE)$ and $M_Y(V,Y)\simeq M_Y(\cE,Y)$.

Step 3. \emph{Proof of $M_Y(V,Y)\simeq M_Y(\cE,X)$.}
Consider
\begin{gather*}
U_i':=\Spec(A[t_i/a,t_1/t_i,\ldots,t_n/t_i],\langle t_i/a,a\rangle),
\\
U_i'':=
\Spec(A[t_i,t_1/t_i,\ldots,t_n/t_i],\langle t_i,a\rangle).
\end{gather*}
Then $(V,Y)=U\cup U_1'\cup \cdots \cup U_n'$, $(\cE,Y)=U_1''\cup \cdots \cup U_n''$, and each $U_i\cup U_i'\to U_i''$ is a dividing cover.
Hence $(V,Y)\to (\cE,Y)$ is a dividing cover,
so
\[
M_Y(V,Y)\simeq M_Y(\cE,Y).
\]

Step 4. \emph{Proof of $M_Y(V)\simeq M_Y(\cE)$.}
Consider
\begin{gather*}
W_i:=(A[1/t_i,t_1/t_i,\ldots,t_n/t_i],\langle a,1/t_i\rangle),
\\
W_i':=(A[a/t_i,t_1/t_i,\ldots,t_n/t_i],\langle a,a/t_i\rangle),
\\
W:=W_1\cup \cdots \cup W_n,
\;
W':=W_1'\cup \cdots \cup W_n'.
\end{gather*}
Then
\[
W\cup \cE \simeq Y\times (\P^n,\P^{n-1})\simeq W'\cup \cE(-D),
\]
which implies
\[
M_Y(W\cup \cE)\simeq 1_Y\simeq M_Y(W'\cup \cE(-D)).
\]
Since each $W_i\cup U_i\to W_i'$ is a dividing cover,
$W\cup U\to W'$ is a dividing cover.
Hence
\[
M_Y(W\cup V)
=
M_Y(W\cup U \cup \cE(-D))
\simeq
M_Y(W'\cup \cE(-D))
\simeq
1_Y.
\]
By Zariski descent,
we have
\[
M_Y(W\cup V/V)
\simeq
M_Y(W/W\cap V)
=
M_Y(W/W\cap \cE)
\simeq
M_Y(W\cup \cE/\cE),
\]
where the middle equality comes from $W\cap V=W\cap \cE$.
Since $M_Y(W\cup V)$ and $M_Y(W\cup \cE)$ are contractible,
we have $M_Y(V)\simeq M_Y(\cE)$.
\end{proof}

\section{Log Gysin sequence}
We keep working with $\sT$ in \S \ref{sat}.
Throughout this section,
let $S\in \lSch$,
and let $Z\to X$ be a strict closed immersion in $\lSm_S^\sat$.
The purpose of this section is to upgrade the log Gysin sequence in \cite[Theorem 7.5.4]{logDM} such that we allow a log structure on the base $S$.

We set $\Bl_Z X:=\Bl_{\ul{Z}}\ul{X}\times_{\ul{X}}X$.
The \emph{deformation space} is defined to be
\[
\rD_Z X
:=
\Bl_{Z\times 0}(X\times \square)-\Bl_{Z\times 0}(X\times 0).
\]
The \emph{normal bundle} is defined to be
\[
\rN_Z X
:=
X\times_{\ul{X}} \rN_{\ul{Z}}\ul{X},
\]
where $\rN_{\ul{Z}}\ul{X}$ denotes the usual normal bundle for schemes.

\begin{thm}
\label{Gysin.25}
Let $S\in \lSch$,
and let $Z\to X$ be a strict closed immersion in $\lSm_S^\sat$.
Then the induced morphisms
\begin{equation}
\label{Gysin.0.2}
M_S(X/(X,Z))
\to
M_S(\rD_Z X/(\rD_Z X,Z\times \square))
\leftarrow
M_S(\rN_Z X/(\rN_Z X,Z))
\end{equation}
are isomorphisms.
\end{thm}
\begin{proof}
Let us use the argument in \cite[Theorem 3.4.1]{logGysin} as follows.
We can work Zariski locally on $X$,
so we may assume the existence of $S'$ in Proposition \ref{Gysin.26}.
Replace $S$ by $S'$ to reduce to the case where $Z$ and $X$ are strict smooth over $S$.
By considering the functor $p^*$ for the morphism $p\colon S\to \ul{S}$ removing the log structure,
We further reduce to the case of $S\in \Sch$.
Then \cite[Theorem 7.5.4]{logDM} is precisely what we need.
\end{proof}

Such a technique of upgrading a statement with a base scheme to a statement with a base fs log scheme will be repeatedly used later.

As a consequence,
we have a natural cofiber sequence
\begin{equation}
\label{Gysin.0.3}
M_S(X,Z)
\to
M_S(X)
\xrightarrow{\Gys}
\Th_Z(\rN_Z X),
\end{equation}
which is called the \emph{log Gysin sequence}.
This will play a crucial role throughout this paper.
The morphism $\Gys\colon M_S(X)\to \Th_Z(\rN_Z X)$ is called the \emph{log Gysin morphism}.

The natural transformation $(fu)_\sharp u^* \xrightarrow{\ad'} f_\sharp$ below is not obtained by the counit of an adjoint pair $(u_\sharp,u^*)$ since $u_\sharp$ does not exist, but we use $\ad'$ for convenience.
We will use the same convention for many other places later.

We can also upgrade the log Gysin sequence to natural transformations as follows.

\begin{thm}
\label{setup.1}
Let $u\colon (X,Z)\to X$ be the canonical morphism.
Then there exists a cofiber sequence of natural transformations
\[
(fu)_\sharp u^*
\xrightarrow{\ad'}
f_\sharp
\xrightarrow{\Gys}
(fi)_\sharp \Sigma^{\rN_i} i^*,
\]
which we call the \emph{log Gysin sequence}.
Here, $\rN_i:=\rN_Z X$ for abbreviation.
\end{thm}
\begin{proof}
Consider the induced commutative diagram
\[
\begin{tikzcd}
(X,Z)\ar[d,"u"']\ar[r]&
(\rD_Z X,X\times \square)\ar[r,leftarrow]\ar[d,"u'"]&
(\rN_Z X,Z)\ar[d,"u''"]
\\
X\ar[r,"i_1"]&
\rD_Z X\ar[r,leftarrow,"i_0"]\ar[d,"g"]&
\rN_Z X
\\
&
X.
\end{tikzcd}
\]
We claim that the squares in the induced commutative diagram
\begin{equation}
\label{setup.1.1}
\begin{tikzcd}
(fgi_1u)_\sharp u^*i_1^*g^*\ar[d,"\ad'"']\ar[r,"\ad'"]&
(fgu')_\sharp u'^*g^*\ar[d,"\ad'"]\ar[r,leftarrow,"\ad'"]&
(fgi_0u'')_\sharp u''^*i_0^*g^*\ar[d,"\ad'"]
\\
(fgi_1)_\sharp i_1^*g^*\ar[r,"\ad'"]&
(fg)_\sharp g^*\ar[r,leftarrow,"\ad'"]&
(fgi_0)_\sharp i_0^*g^*
\end{tikzcd}
\end{equation}
are cartesian.
For this,
it suffices to show the same after evaluating at $M_X(V)$ for every $V\in \lSm_X^\sat$,
i.e.,
the squares in
\[
\begin{tikzcd}
M_S(V\times_X (X,Z))\ar[d]\ar[r]&
M_S(V\times_X (\rD_Z X,Z\times \square))\ar[r,leftarrow]\ar[d]&
M_S(V\times_X (\rN_Z X,Z))\ar[d]
\\
M_S(V)\ar[r]&
M_S(V\times_X \rD_Z X)\ar[r,leftarrow]&
M_S(V\times_X \rN_Z X)
\end{tikzcd}
\]
are cartesian.
This follows from Theorem \ref{Gysin.25} since $\rD_W V\simeq V\times_X \rD_Z X$ and $\rN_W V\simeq V\times_X \rN_Z X$ with $W:=V\times_X Z$.

Since $gi_1=\id$,
we obtain a natural isomorphism
\[
\cofib((fu)_\sharp u^*\xrightarrow{\ad'} f_\sharp)
\simeq
\cofib((fgi_0u'')_\sharp u''^*i_0^* g^* \xrightarrow{\ad'} (fgi_0)_\sharp i_0^* g^*).
\]
To conclude,
observe that the right-hand side is isomorphic to $(fi)_\sharp \Sigma^{\rN_i}i^*$ by the $\lSm^\sat$ projection formula.
\end{proof}

\begin{const}
\label{Gysin.10}
Under the above notation,
assume that $\cE$ is a vector bundle over $X$ such that $i^*\cE\simeq \rN_i$.
Theorem \ref{setup.1} yields a cofiber sequence
\[
(fu)_\sharp u^* \Sigma^{-\cE}
\to
f_\sharp \Sigma^{-\cE}
\xrightarrow{\Gys}
(fi)_\sharp i^*.
\]
In particular,
we obtain a morphism
\[
\Th_X^{-1}(\cE)
\xrightarrow{\Gys}
M_S(Z).
\]
\end{const}

\section{Log Gysin morphisms are compatible with compositions}
\label{Gysin}

We keep working with $\sT$ in \S \ref{sat}.
D\'eglise \cite[Theorem 4.32]{MR2466188} proved that Gysin morphisms are compatible with compositions in the $\A^1$-invariant motivic setting.
In Theorem \ref{Gysin.11},
we argue similarly to prove that the same conclusion holds without the $\A^1$-invariance.

Throughout this section,
let $S\in \lSch$,
let $Z\xrightarrow{j} Y\xrightarrow{i} X$ be closed immersions in $\lSm_S^\sat$,
and let $f\colon X\to S$ be the structural morphism.

\begin{prop}
\label{Gysin.27}
The canonical morphism
\[
M_S((X,Z),(Y,Z)) \to
M_S(X,Y)
\]
is an isomorphism.
\end{prop}
\begin{proof}
As in Theorem \ref{Gysin.25},
we reduce to the case where $S\in \Sch$ and $X,Y,Z\in \Sm_S$.
We have
\begin{gather*}
\ul{(X,Y)}\simeq \Bl_Y X,
\;
(X,Y)-\partial (X,Y)\simeq X-Y,
\\
\ul{((X,Z),(Y,Z))}
=
\Bl_{\Bl_Z Y} (\Bl_Z X)
\simeq
\Bl_{Z\times_X \Bl_Y X} (\Bl_Y X),
\\
((X,Z),(Y,Z))-\partial ((X,Z),(Y,Z))
\simeq X-Y.
\end{gather*}
This shows that $((X,Z),(Y,Z))\to (X,Y)$ is an admissible blow-up along a smooth center,
so \cite[Theorem 7.2.10]{logDM} finishes the proof.
\end{proof}

\begin{prop}
\label{Gysin.28}
We have a natural cofiber sequence
\[
M_S((X,Z),(Y,Z)) \to M_S(X,Z) \to \Th_{(Y,Z)}((Y,Z)\times_Y \rN_Y X).
\]
\end{prop}
\begin{proof}
Using \eqref{Gysin.0.3},
it suffices to construct a natural isomorphism
\[
\Th_{(Y,Z)}((Y,Z)\times_Y \rN_Y X)
\simeq
\Th_{(Y,Z)}(\rN_{(Y,Z)}(X,Z)).
\]
As in Theorem \ref{Gysin.25},
we reduce to the case where $S\in \Sch$ and $X,Y,Z\in \Sm_S$.

The total transform $Y\times_X \Bl_Z X$ is the union of the strict transform $\Bl_Y X$ and the exceptional divisor $E$.
Hence we have
\[
p^*(\cI/\cI^2)\simeq (\cJ/\cJ^2)(E),
\]
where $p\colon \Bl_Z Y\to Y$ is the projection,
and $\cI$ (resp.\ $\cJ$) is the sheaf of ideals defining $Y$ (resp.\ $\Bl_Z Y$) in $X$ (resp.\ $\Bl_Z X$).
Since $\rN_Y X$ (resp.\ $\rN_{\Bl_Z Y}(\Bl_Z X)$) is the vector bundle associated with the dual of $\cI/\cI^2$ (resp.\ $\cJ/\cJ^2$),
we have an isomorphism of vector bundles
\[
(\Bl_Z Y\times_Y \rN_Y X)(E)
\simeq
\rN_{\Bl_Z Y}(\Bl_Z X).
\]
We also have
\[
(Y,Z)\times_{\Bl_Z Y} \rN_{\Bl_Z Y}(\Bl_Z X)
\simeq
\rN_{(Y,Z)}(X,Z),
\]
so Proposition \ref{blow.5} finishes the proof.
\end{proof}

\begin{const}
\label{Gysin.8}
Let $\cE$ be a vector bundle on $X$.
We have the Gysin isomorphisms
\[
M_S(\cE/(\cE\times_X (X,Y)))
\simeq
M_S((i^*\cE\times_Y \rN_Y X) / (i^*\cE\times_Y (\rN_Y X,Y))),
\]
\begin{align*}
&M_S((\cE,X)/((\cE,X)\times_X (X,Y)))
\\
\simeq &
M_S(((i^*\cE,Y)\times_Y \rN_Y X) / ((i^*\cE,Y)\times_Y (\rN_Y X,Y))).
\end{align*}
On the other hand,
we have a natural isomorphism
\[
M_Y(\Th_Z(i^*\cE))\otimes_Y M_Y(\Th_Y(\rN_Y X))
\simeq
M_Y(i^*\cE/(i^*\cE,Y)) \otimes_Y M_Y(\rN_Y X/(\rN_Y X,Y)).
\]
Combining these,
we obtain a natural cofiber sequence
\begin{equation}
\label{Gysin.8.1}
\Th_{(X,Y)}(\cE\times_X (X,Y))
 \to 
\Th_{X}(\cE)
\xrightarrow{\Gys} 
(fi)_\sharp (\Th_Y(i^*\cE)\otimes_Y \Th_Y(\rN_Y X)).
\end{equation}
\end{const}

This allows to show the multiplicative property of Thom spaces as follows. 

\begin{prop}
\label{Gysin.9}
Let $0\to \cE'\to \cE\to \cE''\to 0$ be an exact sequence of vector bundles over $X$.
Then there exists a natural isomorphism
\begin{equation}
\label{Gysin.9.1}
\Th_X(\cE)
\simeq
\Th_X(\cE')\otimes_X \Th_X(\cE'').
\end{equation}
\end{prop}
\begin{proof}
Propositions \ref{Gysin.27} and \ref{Gysin.28} yield a commutative triangle
\[
\begin{tikzcd}
M_X(\cE,\cE')\ar[rr]\ar[rd]& &
M_X(\cE,X)\ar[ld]
\\
&M_X(\cE)
\end{tikzcd}
\]
with natural isomorphisms
\begin{gather*}
M_X(\cE/(\cE,\cE'))\simeq \Th_{\cE'}(\rN_{\cE'} \cE),
\\
M_X((\cE,X)/(\cE,\cE'))
\simeq
\Th_{(\cE',X)}((\cE',X)\times_{\cE'} \rN_{\cE'} \cE),
\end{gather*}
where $X$ in $(\cE,X)$ is regarded as the zero section.
Hence we have natural isomorphisms
\[
M_X(\cE/(\cE,X))
\simeq
M_X \left(\frac{\cE/(\cE,\cE')}{(\cE,X)/(\cE,\cE')}\right)
\simeq
\frac{\Th_{\cE'}(\rN_{\cE'}\cE)}{\Th_{(\cE',X)}((\cE',X)\times_{\cE'}\rN_{\cE'}\cE)}.
\]
Let $p\colon \cE'\to X$ be the projection.
Applying Construction \ref{Gysin.8} to the case of $(Y,X,\cE)=(X,\cE',\rN_{\cE'} \cE)$ yields
\[
\frac{\Th_{\cE'}(\rN_{\cE'}\cE)}{\Th_{(\cE',X)}((\cE',X)\times_{\cE'}\rN_{\cE'}\cE)}
\simeq
\Th_X(i_0^*\rN_{\cE'}\cE)\otimes_X \Th_X(\cE'),
\]
where $i_0\colon X\to \cE'$ is the zero section.
Use $\rN_{\cE'}\cE\simeq p^*\cE''$ and $M_X(\cE/(\cE,X))\simeq \Th_X(\cE)$ to conclude.
\end{proof}

By Proposition \ref{Gysin.9},
we can write \eqref{Gysin.8.1} as a natural cofiber sequence
\begin{equation}
\label{Gysin.9.2}
\Th_{(X,Z)}((X,Z)\times_X \cE)
\to
\Th_{X}(\cE)
\xrightarrow{\Gys}
\Th_Z(i^*\cE\oplus \rN_Z X).
\end{equation}

Now,
we prove the main result of this section.

\begin{thm}
\label{Gysin.11}
There exists a natural commutative diagram
\[
\begin{tikzcd}[row sep=small]
M_S(X/(X,Y))\ar[r,"\Gys","\simeq"']\ar[d]&
\Th_Y(\rN_Y X)\ar[dd,"\Gys"]
\\
M_S(X/(X,Z))\ar[d,"\Gys"',"\simeq"]
\\
\Th_Z(\rN_Z X)\ar[r,"\simeq"]
&
\Th_Z(j^*\rN_Y X\oplus \rN_Z Y),
\end{tikzcd}
\]
where the right vertical morphism is obtained by \eqref{Gysin.9.2},
and the lower horizontal isomorphism is obtained by \eqref{Gysin.9.1} and the canonical exact sequence of vector bundles
\[
0\to \rN_Z Y \to \rN_Z X \to j^* \rN_Y X \to 0.
\]
\end{thm}
\begin{proof}
The composite $M_S(X/(X,Y))\to M_S(\Th_Z(j^*\rN_Y X\oplus \rN_Z Y))$ along the upper right path admits a factorization
\begin{align*}
M_S(X/(X,Y))
\to
M_S\left(\frac{X/(X,Y)}{(X,Z)/(X,Y)}\right)
\xrightarrow{\Gys} &
\frac{\Th_Y(\rN_Y X)}{\Th_{(Y,Z)}(\rN_Y X)}
\\
\xrightarrow{\Gys} &
\Th_Z(j^*\rN_Y X \oplus \rN_Z Y).
\end{align*}
Hence it suffices to construct a commutative square
\[
\begin{tikzcd}
M_S(X/(X,Z))\ar[r,"\simeq"]\ar[d,"\simeq"']&
M_S\left(\frac{X/(X,Y)}{(X,Z)/(X,Y)}\right)\ar[d,"\simeq"]
\\
\Th_Z(\rN_Z X)\ar[r,"\simeq"]&
\Th_Z(j^*\rN_Y X \oplus \rN_Z Y).
\end{tikzcd}
\]
Consider
\begin{gather*}
D:=\rD(X,Y,Z):=\rD_{Y\times_X \rD_Z X}(\rD_Z X),
\\
D':=\rD(Y,Y,Z)\simeq \rD_Z Y\times \square,
\text{ }
D'':=\rD(Z,Z,Z)\simeq Z\times \square^2,
\\
N:=\rN(X,Y,Z):=\rN_{\rN_Z Y}(\rN_Z X)
\simeq
j^*\rN_Y X \oplus \rN_Z Y,
\\
N':=\rN(Y,Y,Z)\simeq \rN_Z Y,
\text{ }
N'':=\rN(Z,Z,Z)\simeq Z.
\end{gather*}
We have a commutative diagram with horizontal isomorphisms
\[
\begin{tikzcd}
M_S(X/(X,Z))\ar[d,"\simeq"']\ar[r,"\simeq"]&
M_S(D/(D,D''))\ar[d,"\simeq"]\ar[r,leftarrow,"\simeq"]&
M_S(N/(N,N''))\ar[d,"\simeq"]
\\
\Th_Z(\rN_Z X)\ar[r]&
\Th_{D''}(\rN_{D''} D)\ar[r,leftarrow]&
\Th_{N''}(\rN_{N''} N),
\end{tikzcd}
\]
where we will show that the upper horizontal morphisms are isomorphisms in Proposition \ref{Gysin.7} below.
We similarly have a commutative diagram with horizontal isomorphisms
\[
\begin{tikzcd}[column sep=tiny]
M_S \left(\frac{X/(X,Y)}{(X,Z)/(X,Y)}\right)\ar[r,"\simeq"]\ar[d,"\simeq"']&
M_S\left(\frac{D/(D,D')}{(D,D'')/(D,D')}\right)\ar[r,leftarrow,"\simeq"]\ar[d,"\simeq"]&
M_S\left(\frac{N/(N,N')}{(N,N'')/(N,N')}\right)\ar[d,"\simeq"]
\\
\Th_Z(j^*\rN_Y X\oplus \rN_Z Y)\ar[r]&
\Th_{D''}(j^*\rN_{D'} D\oplus \rN_{D''} D)\ar[r,leftarrow]&
\Th_{N''}(j^*\rN_{N'} D\oplus \rN_{N''} N).
\end{tikzcd}
\]
Hence we can replace $(X,Y,Z)$ by $(N,N',N'')$.
In this case,
compare with the construction of \eqref{Gysin.9.1} to conclude.
\end{proof}

\begin{prop}
\label{Gysin.7}
Under the above notation,
the induced morphisms
\begin{align*}
M_S(X/(X,Z))
\to &
M_S(\rD(X,Y,Z)/(\rD(X,Y,Z),Z\times \square^2))
\\
\leftarrow &
M_S(\rN(X,Y,Z)/(\rN(X,Y,Z),Z))
\end{align*}
are isomorphisms.
\end{prop}
\begin{proof}
Argue as in Theorem \ref{Gysin.25} to reduce to the case where $S\in \Sch$ and $X,Y,Z\in \Sm_S$.
Observe that $\rD(X,Y,Z)$ admits a natural morphism to $\square^2$ whose fiber at $(1,1)$ (resp.\ $(0,0)$) is $X/(X,Z)$ (resp.\ $\rN(X,Y,Z)$).
We can work Zariski locally on $X$,
so we may assume that there exists a commutative diagram with cartesian squares
\[
\begin{tikzcd}
Z\ar[d]\ar[r]&
Y\ar[d]\ar[r]&
X\ar[d]
\\
\A^c
\ar[r]&
\A^b\ar[r]&
\A^a
\end{tikzcd}
\]
with integers $a\geq b\geq c\geq 0$ such that $X\to \A^a$ is separated and \'etale and $\A^c\to \A^b\to \A^a$ are the usual closed immersions.
Let us adapt \cite[Lemma 2.28 in \S 3]{MV} for the triple $(X,Y,Z)$.
We have an induced commutative diagram with cartesian squares
\[
\begin{tikzcd}[row sep=small]
Z\ar[r,leftarrow]\ar[d]&
Z\times_{\A^c}Z\ar[r]\ar[d]&
Z\ar[d]
\\
Y\ar[r,leftarrow]\ar[d]&
Y\times_{\A^b}(Z\times \A^{b-c})\ar[r]\ar[d]&
Z\times \A^{b-c}\ar[d]
\\
X\ar[r,leftarrow]&
X\times_{\A^a}(Z\times \A^{a-c})\ar[r]&
Z\times \A^{a-c}.
\end{tikzcd}
\]
Since $Z\to \A^c$ is separated and \'etale, $Z\times_{\A^c} Z$ is the disjoint union of the diagonal $\Delta$ and $C:=Z\times_{\A^c} Z-\Delta$.
Note that $C$ is closed in $X\times_{\A^a} (Z\times \A^{a-c})$.
With $X':=X\times_{\A^a} (Z\times \A^{a-c})-C$ and $Y':=Y\times_{\A^b} (Z\times \A^{b-c})-C$,
we have an induced commutative diagram with cartesian squares
\begin{equation}
\begin{tikzcd}[row sep=small]
Z\ar[r,leftarrow,"\id"]\ar[d]&
Z\ar[r,"\id"]\ar[d]&
Z\ar[d]
\\
Y\ar[r,leftarrow]\ar[d]&
Y'\ar[r]\ar[d]&
Z\times \A^{b-c}\ar[d]
\\
X\ar[r,leftarrow]&
X'\ar[r]&
Z\times \A^{a-c}
\end{tikzcd}
\end{equation}
such that $X'\to X$ and $X'\to Z\times \A^{a-c}$ are \'etale.
We set $e:=b-c$ and $d:=a-b$.
By \cite[Proposition 3.2.17]{logSH}.
we have isomorphisms
\[
M_S(X/(X,Z))
\xleftarrow{\simeq}
M_S(X'/(X',Z))
\xrightarrow{\simeq}
M_S(Z\times \A^{e+d} /(Z\times \A^{e+d},Z))
\]
and similarly for $\rD(X,Y,Z)$ and $\rN(X,Y,Z)$.
Hence we may assume $(X,Y,Z)=(Z\times \A^{e+d},Z\times \A^e,Z)$.

With this reduction,
let us argue as in \cite[Theorem 7.5.4]{logDM} and \cite[Lemma 3.2.18]{logSH}.
We can write $\rD(X,Y,Z)$ as the gluing of
\begin{gather*}
Z[y_1t^{-1},\ldots,y_et^{-1},x_1(tu)^{-1},\ldots,x_d(tu)^{-1},t,u],
\\
(Z[y_1t^{-1},\ldots,y_et^{-1},x_1t^{-1},\ldots,x_dt^{-1},t,u^{-1}],\langle u^{-1}\rangle)
\\
(Z[y_1,\ldots,y_e,x_1u^{-1},\ldots,x_du^{-1},t^{-1},u],\langle t^{-1}\rangle ),
\\
(Z[y_1,\ldots,y_e,x_1,\ldots,x_d,t^{-1},u^{-1}], \langle t^{-1},u^{-1}\rangle).
\end{gather*}

Consider the gluing $V\cong Z\times \A^{e+d}\times \square^2$ of
\begin{gather*}
Z[y_1t^{-1},\ldots,y_et^{-1},x_1(tu)^{-1},\ldots,x_d(tu)^{-1},t,u],
\\
(Z[y_1t^{-1},\ldots,y_et^{-1},x_1(tu)^{-1},\ldots,x_d(tu)^{-1},t,u^{-1}],\langle u^{-1}\rangle ),
\\
(Z[y_1t^{-1},\ldots,y_et^{-1},x_1(tu)^{-1},\ldots,x_d(tu)^{-1},t^{-1},u],\langle t^{-1}\rangle),
\\
(Z[y_1t^{-1},\ldots,y_et^{-1},x_1(tu)^{-1},\ldots,x_d(tu)^{-1},t^{-1},u^{-1}],\langle t^{-1},u^{-1}\rangle).
\end{gather*}
The blow-up of $\ul{V}$ along $Z\times (0\in \A^{e+d}) \times \infty \times \P^1$ contains the gluing of
\begin{gather*}
Z[y_1t^{-1},\ldots,y_et^{-1},x_1(tu)^{-1},\ldots,x_d(tu)^{-1},t,u],
\\
Z[y_1t^{-1},\ldots,y_et^{-1},x_1(tu)^{-1},\ldots,x_d(tu)^{-1},t,u^{-1}],
\\
Z[y_1,\ldots,y_e,x_1u^{-1},\ldots,x_du^{-1},t^{-1},u],
\\
Z[y_1,\ldots,y_e,x_1u^{-1},\ldots,x_du^{-1},t^{-1},u^{-1}]
\end{gather*}
as an open subscheme.
The further blow-up $\ul{V'}$ along the closure of
\[
Z[y_1,\ldots,y_e,x_1u^{-1},\ldots,x_du^{-1},t^{-1},u^{-1}]/(x_1u^{-1},\ldots,x_du^{-1},u^{-1})
\]
contains $\ul{\rD(X,Y,Z)}$ as an open subscheme.
Consider $V'$ with the underlying scheme $\ul{V'}$ equipped with the log structure such that $V'\to V$ is an admissible blow-up along a smooth center.
Let $W'$ be the strict transform of $Z\times \square^2= Z \times (0\in \A^{e+d}) \times \square^2$ in $V'$.
Then the induced morphism $(V',W')\to (V,Z\times \square^2)$ is an admissible blow-up along a smooth center,
so we have an isomorphism
\begin{equation}
\label{Gysin.7.1}
M_S(V'/(V',W'))
\simeq
M_S(V/(V\times \square^2)).
\end{equation}
On the other hand,
we can naturally associate an open subscheme $U'$ of $V'$ from the local description of the blow-up such that $\rD(X,Y,Z)$ and $U'$ cover $V'$ and $\rD(X,Y,Z)\cap U'\cap (Z\times \square^2)=\emptyset$.
By Zariski descent,
we have an induced cocartesian square
\[
\begin{tikzcd}
M_S(U'\cap \rD(X,Y,Z)/U'\cap \rD(X,Y,Z))\ar[r]\ar[d]&
M_S(U'/ U')\ar[d]
\\
M_S(\rD(X,Y,Z)/(\rD(X,Y,Z),Z\times \square^2))\ar[r]&
M_S(V'/(V',W')).
\end{tikzcd}
\]
Hence we have an isomorphism
\begin{equation}
\label{Gysin.7.2}
M_S(\rD(X,Y,Z)/(\rD(X,Y,Z),Z\times \square^2))
\simeq
M_S(V'/(V',W')).
\end{equation}
With \eqref{Gysin.7.1} and \eqref{Gysin.7.2} in hand,
the two morphisms
\begin{align*}
& M_S(X/(X,Z)),M_S(\rN(X,Y,Z)/(\rN(X,Y,Z),Z))
\\
\rightrightarrows &
M_S(\rD(X,Y,Z)/(\rD(X,Y,Z),Z\times \square^2))
\end{align*}
can be identified with the $(1,1)$ and $(0,0)$ sections
\begin{align*}
&M_S(Z\times \A^{d+e} / (Z\times \A^{d+e},Z))
\\
 \rightrightarrows &
M_S(Z\times \A^{d+e} \times \square^2
/
(Z\times \A^{d+e} \times \square^2,Z\times \square^2)).
\end{align*}
We conclude by the $\square$-invariance.
\end{proof}

The following is the natural transformation version of Theorem \ref{Gysin.11}.

\begin{thm}
\label{Gysin.12}
Let $u\colon (X,Y)\to X$ be the canonical morphism.
Then there exists a commutative cube of functors
\[
\begin{tikzcd}[column sep=tiny, row sep=tiny]
&
f_\sharp u_\sharp u^*\ar[ld]\ar[rr]\ar[dd]&
&
0\ar[ld]\ar[dd]
\\
0\ar[rr,crossing over]\ar[dd]&
&
0
\\
&
f_\sharp\ar[rr,"\Gys",near start]\ar[ld,"\Gys"']&
&
(fi)_\sharp \Sigma^{\rN_Y X}i^*\ar[ld,"\Gys"]
\\
(fij)_\sharp \Sigma^{\rN_Z X}(ij)^*\ar[rr,"\simeq"]&
&
(fij)_\sharp \Sigma^{\rN_Z Y}j^*\Sigma^{\rN_Y X} i^*\ar[uu,crossing over,leftarrow]
\end{tikzcd}
\]
such that the back and left squares are obtained by the log Gysin sequences for $Y\to X$ and $Z\to X$.
\end{thm}
\begin{proof}
Argue as in Theorem \ref{setup.1} to upgrade Theorem \ref{Gysin.11}.
\end{proof}

\section{A few commutative diagrams}
\label{few}

We keep working with $\sT$ in \S \ref{sat}.
In this section, we will construct a few useful commutative diagrams regarding Gysin morphisms and Thom spaces.

Let $f\colon X\to S$ be a strict smooth separated morphism of fs log schemes.
Consider the commutative diagram
\begin{equation}
\label{Gysin.15.1}
\begin{tikzcd}
X\ar[r,"a"]&
X\times_S X\ar[d,"p_2"']\ar[r,"p_1"]&
X\ar[d,"f"]
\\
&
X\ar[r,"f"]&
S,
\end{tikzcd}
\end{equation}
where $a$ is the diagonal morphism,
and $p_1$ and $p_2$ are the first and second projections.
We have a composite natural transformation
\[
\mathfrak{p}_f
\colon
f_\sharp
\xrightarrow{\simeq}
f_\sharp p_{1*}a_*
\xrightarrow{\Ex}
f_* p_{2\sharp} a_*
\xrightarrow{\Gys}
f_* \Sigma^{\Omega_f^1}a^*a_*
\xrightarrow{\ad'}
f_*\Sigma^{\Omega_f^1},
\]
where the third morphism is obtained by the fact that 
the vector bundle $\rN_a$ is associated with the locally free sheaf $\Omega_f^1$.

\begin{lem}
\label{Thom.1}
Let
\[
\begin{tikzcd}
Z'\ar[d,"g''"']\ar[r,"i'"]&
X'\ar[d,"g'"]\ar[r,"f'"]&
S'\ar[d,"g"]
\\
Z\ar[r,"i"]&
X\ar[r,"f"]&
S
\end{tikzcd}
\]
be a commutative diagram of fs log schemes with cartesian squares.
We set $h:=fi$ and $h'=f'i'$.
Assume that $f$ is saturated log smooth and $i$ is a strict closed immersion.
Then there are natural commutative diagrams
\begin{equation}
\label{Thom.1.1}
\begin{tikzcd}
f_\sharp'g'^*\ar[d,"\Ex"']\ar[r,"\Gys"]&
h_\sharp' \Sigma^{\rN_{i'}}i'^*g'^*\ar[r,"\simeq"]&
h_\sharp' g''^*\Sigma^{\rN_i}i^*\ar[d,"\Ex"]
\\
g^*f_\sharp \ar[rr,"\Gys"]&
&
g^*h_\sharp\Sigma^{\rN_i}i^*,
\end{tikzcd}
\end{equation}
\begin{equation}
\label{Thom.1.2}
\begin{tikzcd}[column sep=small]
f_\sharp g_*'\ar[d,"\Ex"']\ar[r,"\Gys"]&
h_\sharp \Sigma^{\rN_i} i^*g_*'\ar[r,"\Ex"]&
h_\sharp \Sigma^{\rN_i} g_*''i'^*\ar[r,"\simeq"]&
h_\sharp g_*''\Sigma^{\rN_{i'}}i'^*\ar[d,"\Ex"]&
\\
g_* f_\sharp' \ar[rrr,"\Gys"]&
&
&
g_*h_\sharp' \Sigma^{\rN_{i'}}i'^*.
\end{tikzcd}
\end{equation}
\end{lem}
\begin{proof}
For every $V\in \lSm_X^\sat$ with $W:=V\times_X Z$,
we have a natural commutative square
\[
\begin{tikzcd}
M_{S'}(V\times_S S')\ar[r,"\Gys"]\ar[d,"\simeq"']&
\Th_{S'} (\rN_{W\times_S S'} (V\times_S S'))\ar[d,"\simeq"]
\\
g^*M_{S}(V)\ar[r,"\Gys"]&
g^*\Th_S (\rN_W V).
\end{tikzcd}
\]
Arguing as in Theorem \ref{setup.1},
we obtain \eqref{Thom.1.1}.
The diagram
\[
\begin{tikzcd}[column sep=small, row sep=small]
f_\sharp g_*'\ar[r,"\ad"]\ar[dd,"\Gys"']&
g_*g^*f_\sharp g_*'\ar[r,"\Ex^{-1}"]\ar[dd,"\Gys"]\ar[rdd,"(a)",phantom]&
g_*f_\sharp' g'^*g_*'\ar[r,"\ad'"]\ar[d,"\Gys"]&
g_*f_\sharp'\ar[dddd,"\Gys"]
\\
&
&
g_*h_\sharp' \Sigma^{\rN_{i'}}i'^*g'^*g_*'\ar[d,"\simeq"]\ar[rddd,"\ad'",bend left=20]
\\
h_\sharp \Sigma^{\rN_i} i^*g_*'\ar[r,"\ad"]\ar[d,"\Ex"']&
g_*g^*h_\sharp \Sigma^{\rN_i} i^* g_*'\ar[d,"\Ex"]\ar[r,"\Ex^{-1}"]&
g_*h_\sharp' g''^*\Sigma^{\rN_{i}}i^*g_*'\ar[d,"\Ex"]
\\
h_\sharp \Sigma^{\rN_i} g_*''i'^*\ar[r,"\ad"]\ar[d,"\Ex"']&
g_*g^*h_\sharp \Sigma^{\rN_i} g_*''i'^*\ar[d,"\Ex"]\ar[r,"\Ex^{-1}"]&
g_*h_\sharp' g''^*\Sigma^{\rN_{i}}g_*''i'^*\ar[d,"\Ex"]
\\
h_\sharp g_*''\Sigma^{\rN_{i'}}i'^*\ar[r,"\ad"]&
g_*g^* h_\sharp g_*''\Sigma^{\rN_{i'}}i'^*\ar[r,"\Ex^{-1}"]&
g_*h_\sharp' g''^*g_*''\Sigma^{\rN_{i'}}i'^*\ar[r,"\ad'"]&
g_*h_\sharp' \Sigma^{\rN_{i'}}i'^*
\end{tikzcd}
\]
commutes since the small diagram $(a)$ commutes by \eqref{Thom.1.1},
so we obtain \eqref{Thom.1.2}.
\end{proof}

\begin{prop}
\label{Thom.2}
Let
\[
\begin{tikzcd}
X'\ar[d,"f'"']\ar[r,"g'"]&
X\ar[d,"f"]
\\
S'\ar[r,"g"]&
S
\end{tikzcd}
\]
be a cartesian square of fs log schemes.
If $f$ is strict smooth separated,
then there exist commutative diagrams
\[
\begin{tikzcd}
f_\sharp'g'^*\ar[d,"Ex"']\ar[r,"\mathfrak{p}_{f'}"]&
f_*'\Sigma^{\Omega_{f'}^1}g'^*\ar[r,"\simeq"]&
f_*'g'^*\Sigma^{\Omega_f^1}\ar[d,"Ex",leftarrow]
\\
g^*f_\sharp\ar[rr,"\mathfrak{p}_f"]&
&
g^*f_*\Sigma^{\Omega_f^1},
\end{tikzcd}
\quad
\begin{tikzcd}
f_\sharp g_*'\ar[d,"\Ex"']\ar[r,"\mathfrak{p}_f"]&
f_* \Sigma^{\Omega_f^1} g_*'\ar[d,"\simeq"]&
\\
g_*f_\sharp'\ar[r,"\mathfrak{p}_f'"]&
g_*f_*'\Sigma^{\Omega_{f'}^1}.
\end{tikzcd}
\]
Hence if $\mathfrak{p}_f$ and $\mathfrak{p}_f'$ are isomorphisms,
then the natural transformations
\[
\Ex\colon g^*f_*\to f_*'g'^*,
\;
\Ex\colon f_\sharp g_*'\to g_*f_\sharp'
\]
are isomorphisms.
\end{prop}
\begin{proof}
Consider the induced commutative diagram with cartesian squares
\[
\begin{tikzcd}
X'\ar[r,"a'"]\ar[d,"g'"']&
X'\times_{S'} X'\ar[r,"p_2'"]\ar[d,"g''"]&
X'\ar[d,"g'"]
\\
X\ar[r,"a"]&
X\times_S X\ar[r,"p_2"]&
X,
\end{tikzcd}
\]
where $p_2$ is the second projection,
and $a$ is the diagonal morphism.
We have a commutative diagram
\[
\begin{tikzcd}[row sep=small]
f_\sharp'g'^*\ar[r]\ar[ddd,"\Ex"']\ar[rddd,"(a)",phantom]&
f_*' p_{2\sharp}' a_*'g'^*\ar[r,"\Gys"]\ar[d,"\Ex",leftarrow]&
f_*' \Sigma^{\Omega_{f'}^1}a'^*a_*'g'^*\ar[r,"\ad'"]\ar[d,"\Ex",leftarrow]&
f_*'\Sigma^{\Omega_{f'}^1}g'^*\ar[dd,"\simeq",leftarrow]
\\
&
f_*' p_{2\sharp}' g''^*a_*\ar[r,"\Gys"]\ar[d,"\Ex^{-1}",leftarrow]\ar[rd,"(b)",phantom]&
f_*'\Sigma^{\Omega_{f'}^1}a'^*g''^*a_*\ar[d,"\simeq",leftarrow]
\\
&
f_*' g'^* p_{2\sharp}a_*\ar[r,"\Gys"]\ar[d,"\Ex",leftarrow]&
f_*'g'^*\Sigma^{\Omega_f^1}a^*a_*\ar[r,"\ad'"]\ar[d,"\Ex",leftarrow]&
f_*'g'^*\Sigma^{\Omega_f^1}\ar[d,"\Ex",leftarrow]
\\
g^*f_\sharp\ar[r]&
g^*f_*p_{2\sharp} a_*\ar[r,"\Gys"]&
g^*f_*\Sigma^{\Omega_f^1}a^*a_*\ar[r,"\ad'"]&
g^*f_*\Sigma^{\Omega_f^1},
\end{tikzcd}
\]
where the small diagram $(a)$ is obtained as in \cite[Proposition 1.7.7]{Ayo071},
and the small diagram (b) is due to Lemma \ref{Thom.1}.
We also have a commutative diagram
\[
\begin{tikzcd}[row sep=small]
f_\sharp g_*'\ar[r]\ar[dddd,"\Ex"']\ar[rdddd,phantom,"(c)"]&
f_* p_{2\sharp} a_* g_*'\ar[r,"\Gys"]\ar[d,"\simeq"]&
f_* \Sigma^{\Omega_f^1} a^*a_*g_*'\ar[r,"ad'"]\ar[d,"\simeq"]&
f_* \Sigma^{\Omega_f^1} g_*'\ar[ddd,"\simeq"]
\\
&
f_* p_{2\sharp} g_*''a_*'\ar[r,"\Gys"]\ar[dd,"\Ex"]\ar[rdd,phantom,"(d)"]&
f_* \Sigma^{\Omega_f^1} a^* g_*''a_*'\ar[d,"\Ex"]
\\
&
&
f_* \Sigma^{\Omega_f^1} g_*'a'^*a_*'\ar[d,"\simeq"]\ar[ruu,"\ad'",bend right=20]
\\
&
f_* g_*'p_{2\sharp}'a_*'\ar[r,"\Gys"]\ar[d,"\simeq"]&
f_*g_*'\Sigma^{\Omega_{f'}^1} a'^* a_*'\ar[r,"\ad'"]\ar[d,"\simeq"]&
f_* g_*' \Sigma^{\Omega_{f'}^1}\ar[d,"\simeq"]
\\
g_*f_\sharp'\ar[r]&
g_*f_*p_{2\sharp}'a_*'\ar[r,"\Gys"]&
g_*f_*\Sigma^{\Omega_{f'}^1}a'^*a_*'\ar[r,"\ad'"]&
g_*f_*\Sigma^{\Omega_{f'}^1},
\end{tikzcd}
\]
where the small diagram $(c)$ is obtained as in \cite[Proposition 1.7.5]{Ayo071},
and the small diagram (d) is due to Lemma \ref{Thom.1}.
The outer diagrams are the desired ones.
The last claim in the statement is obvious.
\end{proof}

\begin{prop}
\label{Thom.3}
Let $S\in \lSch$,
let $i\colon Z\to X$ be a strict closed immersion of strict smooth separated fs log schemes over $S$,
and let $f\colon X\to S$ be the structural morphism.
Then there exists a commutative diagram
\[
\begin{tikzcd}
f_\sharp i_*\ar[d,"\mathfrak{p}_f"']\ar[r,"\Gys"]&
(fi)_\sharp \Sigma^{\rN_i}i^*i_*\ar[r,"\ad'"]&
(fi)_\sharp \Sigma^{\rN_i}\ar[d,"\mathfrak{p}_{fi}"]
\\
f_*\Sigma^{\Omega_f^1} i_*\ar[rr,"\simeq"]&
&
f_*i_*\Sigma^{\Omega_{fi}^1}\Sigma^{\rN_i}.
\end{tikzcd}
\]
\end{prop}
\begin{proof}
Consider the induced commutative diagram
\[
\begin{tikzcd}[row sep=small, column sep=small]
&
Z\ar[ld,"d"']\ar[rd,"s_2d"]\ar[dd,"i",near start]
\\
Z\times_S Z\ar[rr,crossing over,"s_2",near start]&&
X\times_S Z\ar[rr,"r_2"]\ar[dd,"r_1"]&&
Z\ar[dd,"i"]
\\
&
X\ar[rd,"a"]
\\
Z\times_S X\ar[rr,"q_2"]\ar[d,"q_1"']\ar[uu,leftarrow,near start,"s_1",crossing over]&&
X\times_S X\ar[rr,"p_2"]\ar[d,"p_1"]&&
X\ar[d,"f"]
\\
Z\ar[rr,"i"]&&
X\ar[rr,"f"]&&
S,
\end{tikzcd}
\]
where $p_1$ (resp.\ $p_2$) is the first (resp.\ second) projection, $a$ and $d$ are the diagonal morphisms,
and the small squares are cartesian.
Combine the two commutative diagrams
\[
\begin{tikzcd}[row sep=small]
f_\sharp i_*\ar[r,"\simeq"]\ar[dd,"\Gys"']\ar[rd,"\simeq"]&
f_\sharp p_{1*}a_*i_*\ar[r,"\Ex"]\ar[d,"\simeq"]&
f_*p_{2\sharp} a_* i_*\ar[d,"\simeq"]
\\
&
f_\sharp p_{1*}r_{1*}s_{2*}d_*\ar[r,"\Ex"]\ar[d,"\Gys"]\ar[rd,"\Ex"]&
f_* p_{2\sharp} r_{1*} s_{2*} d_*\ar[d,"\Ex"]
\\
(fi)_\sharp \Sigma^{\rN_i}  i^*i_*\ar[r,"\simeq"]\ar[dddd,"\ad'"']
&
(fi)_\sharp \Sigma^{\rN_i} i^* p_{1*} r_{1*} s_{2*} d_*\ar[d,"\Ex"]\ar[rd,"(a)",phantom]&
f_* i_* r_{2\sharp} s_{2*} d_*\ar[d,"\Gys"]
\\
&
(fi)_\sharp \Sigma^{\rN_i}  q_{1*} s_{1*} s_2^*s_{2*} d_*\ar[d,"\simeq"]&
f_*i_* (r_2s_2)_\sharp \Sigma^{\rN_{s_2}} s_2^* s_{2*} d_*\ar[d,"\ad'"]
\\
&
(fi)_\sharp q_{1*} s_{1*} \Sigma^{\rN_{s_2}}  s_2^*s_{2*} d_*\ar[d,"\ad'"]\ar[ru,"\Ex"]&
f_*i_* (r_2s_2)_\sharp \Sigma^{\rN_{s_2}}  d_*\ar[dd,"\simeq"]
\\
&
(fi)_\sharp q_{1*} s_{1*} \Sigma^{\rN_{s_2}} d_*\ar[d,"\simeq"]\ar[ru,"\Ex"]
\\
(fi)_\sharp \Sigma^{\rN_i}\ar[r,"\simeq"]&
(fi)_\sharp  q_{1*} s_{1*}   d_*\Sigma^{\rN_i}\ar[r,"\Ex"]&
f_*i_* (r_2 s_2)_\sharp d_* \Sigma^{\rN_i}
\end{tikzcd}
\]
and
\[
\begin{tikzcd}[row sep=small, column sep=small]
f_*p_{2\sharp} a_* i_*\ar[r,"\Gys"]\ar[d,"\simeq"']&
f_*\Sigma^{\Omega_f^1} a^*a_* i_*\ar[d,"\simeq"]\ar[r,"\ad'"]&
f_* \Sigma^{\Omega_f^1} i_*\ar[ddddd,bend left=55,"\simeq"]
\\
f_* p_{2\sharp} r_{1*} s_{2*} d_*\ar[dd,"\Ex"']\ar[r,"\Gys"]\ar[rdd,"(b)",phantom]&
f_*\Sigma^{\Omega_f^1} a^*r_{1*} s_{2*} d_*\ar[d,"\Ex"]
\\
&
f_* \Sigma^{\Omega_f^1} i_* d^* s_2^* s_{2*} d_*\ar[d,"\simeq"]\ar[r,"\ad'"]&
f_* \Sigma^{\Omega_f^1} i_* d^* d_*\ar[uu,"\ad'"]\ar[d,"\simeq"]
\\
f_* i_* r_{2\sharp} s_{2*} d_*\ar[d,"\Gys"']\ar[r,"\Gys"]\ar[rd,"(c)",phantom]&
f_* i_* \Sigma^{i^*\Omega_f^1} d^*s_2^* s_{2*} d_*\ar[d,"\simeq"]\ar[r,"\ad'"]&
f_* i_* \Sigma^{i^*\Omega_f^1} d^* d_*\ar[ldd,"\simeq",bend left=15]\ar[dd,"\ad'"]
\\
f_*i_* (r_2s_2)_\sharp \Sigma^{\rN_{s_2}} s_2^* s_{2*} d_*\ar[d,"\ad'"']\ar[r,"\Gys"']&
f_* i_* \Sigma^{\Omega_{fi}^1} d^* \Sigma^{\rN_{s_2}} s_2^* s_{2*} d_*\ar[d,"\ad'"]
\\
f_*i_* (r_2s_2)_\sharp \Sigma^{\rN_{s_2}}  d_*\ar[d,"\simeq"']\ar[r,"\Gys"]&
f_*i_*\Sigma^{\Omega_{fi}^1} d^* \Sigma^{\rN_{s_2}} d_*\ar[d,"\simeq"]&
f_* i_* \Sigma^{i^* \Omega_f^1}\ar[d,"\simeq"]
\\
f_*i_* (r_2 s_2)_\sharp d_* \Sigma^{\rN_i}\ar[r,"\Gys"]&
f_*i_* \Sigma^{\Omega_{fi}^1} d^*d_* \Sigma^{\rN_i}\ar[r,"\ad'"]&
f_*i_* \Sigma^{\Omega_{fi}^1} \Sigma^{\rN_i}
\end{tikzcd}
\]
to conclude,
where the commutativity of $(a)$ and $(b)$ (resp.\ $(c)$) follows from Lemma \ref{Thom.1} (resp.\ Theorem \ref{Gysin.12}).
\end{proof}

\begin{prop}
\label{Thom.5}
Let $S\in \lSch$,
and let
\[
\begin{tikzcd}
Z'\ar[r,"i'"]\ar[d,"g'"']&
X'\ar[d,"g"]
\\
Z\ar[r,"i"]&
X
\end{tikzcd}
\]
be a cartesian square in $\lSm_S^\sat$ such that $i$ is a strict closed immersion and $g$ is log smooth,
and let $f\colon X\to S$ be the structural morphism.
Then there exist commutative diagrams
\begin{equation}
\label{Thom.5.1}
\begin{tikzcd}
f_\sharp g_\sharp\ar[r,"\Gys"]\ar[rd,"\Gys"']&
(fgi')_\sharp \Sigma^{\rN_{i'}} i'^*\ar[d,"\simeq"]
\\
&(fi)_\sharp \Sigma^{\rN_i} i^* g_\sharp,
\end{tikzcd}
\end{equation}
\begin{equation}
\label{Thom.5.2}
\begin{tikzcd}
f_\sharp g_\sharp i_*'\ar[r,"\Gys"]\ar[d,"\Ex"']&
(fgi')_\sharp \Sigma^{\rN_{i'}} i'^*i_*'\ar[r,"\ad'"]&
(fgi')_\sharp \Sigma^{\rN_{i'}}\ar[d,"\simeq"]
\\
f_\sharp i_*g_\sharp'\ar[r,"\Gys"]&
(fi)_\sharp \Sigma^{\rN_i} i^*i_* g_\sharp'\ar[r,"\ad'"]&
(fi)_\sharp \Sigma^{\rN_i} g_\sharp'.
\end{tikzcd}
\end{equation}
\end{prop}
\begin{proof}
For every $V'\in \lSm_{X'}^\sat$ with $W':=V'\times_{X'} Z'$,
we have a natural isomorphism of cofiber sequences
\[
\begin{tikzcd}
M_S(V',W')\ar[r]\ar[d,"\id"']&
M_S(V')\ar[r,"\Gys"]\ar[d,"\id"]&
(fgi')_\sharp (\Th_{Z'}(\rN_{Z'}X') \otimes_{Z'} M_{Z'}(W'))\ar[d,"\simeq"]
\\
M_S(V',W')\ar[r]&
M_S(V')\ar[r,"\Gys"]&
(fi)_\sharp (\Th_Z(\rN_Z X) \otimes_Z M_Z(W')).
\end{tikzcd}
\]
Arguing as in Theorem \ref{setup.1},
we obtain \eqref{Thom.5.1}.
The diagram
\[
\begin{tikzcd}[row sep=small]
f_\sharp g_\sharp i_*'\ar[r,"\Gys"]\ar[rd,"\Gys"',bend right=15]\ar[dd,"\Ex"']\ar[rd,"(a)",phantom,bend left=5]&
(fgi')_\sharp \Sigma^{\rN_{i'}}i'^*i_*'\ar[d,"\simeq"]\ar[r,"\ad'"]&
(fgi')_\sharp \Sigma^{\rN_{i'}}\ar[dd,"\simeq"]
\\
&
(fi)_\sharp \Sigma^{\rN_i} i^* g_\sharp i_*'\ar[d,"\Ex"]
\\
f_\sharp i_*g_\sharp'\ar[r,"\Gys"]&
(fi)_\sharp \Sigma^{\rN_i} i^*i_* g_\sharp'\ar[r,"\ad'"]&
(fi)_\sharp \Sigma^{\rN_i} g_\sharp'
\end{tikzcd}
\]
commutes since the small diagram $(a)$ commutes by \eqref{Thom.5.1}, so we obtain \eqref{Thom.5.2}.
\end{proof}

\begin{prop}
\label{Gysin.24}
Let $S\in \lSch$,
let $i\colon Z\to X$ be a strict closed immersion in $\lSm_S^\sat$,
and let $f\colon X\to S$ be the structural morphisms.
Then for all $\cF\in \sT(X)$ and $\cG\in \sT(S)$,
there is a natural commutative diagram
\[
\begin{tikzcd}
f_\sharp (\cF\otimes_X f^*\cG)\ar[r,"\Gys"]\ar[d,"\Ex"']&
(fi)_\sharp \Sigma^{\rN_i} i^* (\cF\otimes_X f^*\cG)\ar[r,"\simeq"]&
(fi)_\sharp (\Sigma^{\rN_i} i^*\cF \otimes_Z i^*f^*\cG)\ar[d,"\Ex"]
\\
f_\sharp \cF\otimes_S \cG\ar[rr,"\Gys"]&
&
(fi)_\sharp \Sigma^{\rN_i} i^* \cF \otimes_S \cG.
\end{tikzcd}
\]
\end{prop}
\begin{proof}
For all $V\in \lSm_X^\sat$ and $W\in \lSm_S^\sat$,
we have a natural commutative square
\[
\begin{tikzcd}
M_S(V\times_S W)\ar[d,"\id"']\ar[r,"\Gys"]&
\Th_{Z\times_X V\times_S W}(\rN_i \times_X V\times_S W)\ar[d,"\simeq"]
\\
M_S(V\times_S W)\ar[r,"\Gys"]&
\Th_{Z\times_X V}(\rN_i \times_X V) \otimes_S M_S(W).
\end{tikzcd}
\]
Arguing as in Theorem \ref{setup.1},
we obtain the desired diagram.
\end{proof}

\begin{prop}
\label{Gysin.23}
Let $f\colon X\to S$ be a strict smooth separated morphism of fs log schemes.
Then for every $\cF\in \sT(X)$ and $\cG\in \sT(S)$,
there exists a natural commutative diagram
\[
\begin{tikzcd}
f_\sharp(\cF\otimes_X f^*\cG)
\ar[r,"\mathfrak{p}_f"]\ar[d,"\Ex"']&
f_* \Sigma^{\Omega_f^1}(\cF\otimes_X f^* \cG)\ar[r,"\simeq"]&
f_* (\Sigma^{\Omega_f^1}\cF\otimes_X f^* \cG)\ar[d,"\Ex",leftarrow]
\\
f_\sharp \cF \otimes_S \cG\ar[rr,"\mathfrak{p}_f"]&
&
f_* \Sigma^{\Omega_f^1} \cF\otimes_S \cG.
\end{tikzcd}
\]
\end{prop}
\begin{proof}
Consider the diagram \eqref{Gysin.15.1}.
As in \cite[Proof of Lemma 2.4.23(3)]{CD19},
we have a commutative diagram
\[
\begin{tikzcd}[column sep=small, row sep=small]
f_\sharp(\cF\otimes_X f^*\cG)\ar[r,"\simeq"]\ar[dddd,"\Ex"']&
f_\sharp p_{1*}a_*(\cF\otimes_X f^*\cG)\ar[r,"\Ex"]&
f_*p_{2\sharp}a_*(\cF\otimes_X f^*\cG)\ar[d,"\simeq",leftarrow]
\\
&
&
f_*p_{2\sharp}a_*(\cF\otimes_X a^*p_2^* f^*\cG)\ar[d,leftarrow,"\Ex"]
\\
&
&
f_*p_{2\sharp}(a_* \cF\otimes_{V} p_2^*f^*\cG)\ar[d,"\Ex^{-1}",leftarrow]
\\
&
&
f_*(p_{2\sharp}a_*\cF\otimes_X f^*\cG)\ar[d,"\Ex",leftarrow]
\\
f_\sharp \cF\otimes_S \cG\ar[r,"\simeq"]&
f_\sharp p_{1*}a_*\cF\otimes_S \cG\ar[r,"\Ex"]&
f_*p_{2\sharp}a_*\cF \otimes_S \cG,
\end{tikzcd}
\]
where $V:=X\times_S X$ for abbreviation.
Combine this with the commutative diagram
\[
\begin{tikzcd}[row sep=small, column sep=tiny]
f_* p_{2\sharp}a_*(\cF \otimes_X f^*\cG)\ar[r,"\Gys"]\ar[d,"\simeq"',leftarrow]&
f_* \Sigma^{\Omega_f^1}a^*a_* (\cF\otimes_X f^*\cG)\ar[r,"\ad'"]\ar[d,"\simeq",leftarrow]&
f_* \Sigma^{\Omega_f^1}(\cF\otimes_X f^*\cG)\ar[ddd,"\simeq",leftarrow]
\\
f_*p_{2\sharp}a_*(\cF\otimes_X a^*p_2^* f^*\cG)\ar[d,leftarrow,"\Ex"']\ar[r,"\Gys"]&
f_*\Sigma^{\Omega_f^1} a^*a_*(\cF\otimes_X a^*p_2^* f^*\cG)\ar[d,"\Ex",leftarrow]
\\
f_*p_{2\sharp}(a_* \cF\otimes_{V} p_2^*f^*\cG)\ar[d,"\Ex^{-1}"',leftarrow]\ar[r,"\Gys"]\ar[rd,"(a)",phantom]&
f_*\Sigma^{\Omega_f^1}a^*(a_*\cF\otimes_V p_2^*f^*\cG)\ar[d,"\simeq",leftarrow]
\\
f_*(p_{2\sharp}a_*\cF\otimes_X f^*\cG)\ar[d,"\Ex"',leftarrow]\ar[r,"\Gys"']&
f_*(\Sigma^{\Omega_f^1}a^*a_*\cF\otimes_X f^*\cG)\ar[r,"\ad'"]\ar[d,"\Ex",leftarrow]&
f_*(\Sigma^{\Omega_f^1}\cF\otimes_X f^*\cG)\ar[d,"\Ex",leftarrow]
\\
f_*p_{2\sharp}a_*\cF \otimes_S \cG\ar[r,"\Gys"]&
f_*\Sigma^{\Omega_f^1}a^*a_*\cF\otimes_S \cG\ar[r,"\ad'"]&
f_*\Sigma^{\Omega_f^1}\cF\otimes_S \cG
\end{tikzcd}
\]
to conclude,
where the commutativity of $(a)$ follows from Proposition \ref{Gysin.24}.
\end{proof}

\section{Evaluation morphism}
\label{ev}

We keep working with $\sT$ in \S \ref{sat}.
Recall that an object $\cF$ of a symmetric monoidal $\infty$-category is called \emph{dualizable} if there exists another object $\cG$ equipped with morphisms
\[
\ev\colon \cF\otimes \cG\to 1,
\;
\mathrm{coev}\colon 1\to \cG\otimes \cF
\]
satisfying the triangle identities.
In this section,
we focus on the construction of a morphism that behaves like $\ev$ for $M_S(X)$ with strict smooth projective $X$ over $S$.
Using this,
we will show in Proposition \ref{Gysin.21} that $\mathfrak{p}_p$ is an isomorphism, where $p\colon \P^n\to S$ is the projection with $S\in \lSch$.

\begin{const}
\label{ev.1}
Let $f\colon X\to S$ be a strict smooth morphism of fs log schemes.
We have the morphism
\[
\ev
\colon
M_S(X) \otimes_S \Th_S^{-1}(\Omega_f^1)
\to
\unit_S,
\]
which is called the \emph{evaluation morphism},
given by the composite
\[
M_S(X)\otimes_S \Th_S^{-1}(\Omega_f^1)
\xrightarrow{\simeq}
\Th_S^{-1}(\Omega_{p_2}^1)
\xrightarrow{\Gys}
M_S(X)
\to
\unit_S,
\]
where $p_2\colon X\times_S X\to X$ is the second projection,
the Gysin morphism is obtained by Construction \ref{Gysin.10} using the diagonal morphism $X\to X\times_S X$,
and the last morphism is induced by $f$.

By adjunction,
we obtain a morphism
\[
\comp
\colon
\Th_X^{-1}(\Omega_f^1)
\to
D(M_S(X)),
\]
which is called the \emph{comparison morphism}.
\end{const}

Let $S\in \Sch$, let $X\in \lSm_S$, and let $\cE\to \cF$ be a morphism of vector bundles over $X$ such that $\cE\times_X (X-\partial X)\to \cF\times_X (X-\partial X)$ is an isomorphism.
We expect that we have an induced isomorphism
\begin{equation}
\Th_X^{-1}(\cF)
\xrightarrow{\simeq}
\Th_X^{-1}(\cE)
\end{equation}
in $\sT(S)$.
We only need the following case later.

\begin{prop}
\label{blow.4}
Let $S\in \lSch$,
and let $Z\to X$ be a closed immersion in $\lSm_S^\sat$.
Then we have an induced isomorphism
\begin{equation}
\label{blow.4.1}
\Th_{(X,Z)}^{-1}(\Omega_{(X,Z)/S}^1)
\xrightarrow{\simeq}
\Th_{(X,Z)}^{-1}(u^*\Omega_{X/S}^1)
\end{equation}
in $\sT(S)$,
where $u\colon (X,Z)\to X$ is the canonical morphism.
\end{prop}
\begin{proof}
As in Theorem \ref{Gysin.25}, we reduce to the case where $S\in \Sch$ and $X,Z\in \Sm_S$.
The question is Zariski local on $X$.
Hence by \cite[Proof of Lemma 2.28 in \S 3]{MV},
we may assume that there exists a commutative diagram of schemes with cartesian squares
\[
\begin{tikzcd}
Z\ar[d]\ar[r,"\id",leftarrow]&
Z\ar[d]\ar[r,"\id"]&
Z\ar[d]
\\
X\ar[r,leftarrow]&
X_1\ar[r]&
X_2
\end{tikzcd}
\]
such that $X_1\to X,X_2$ are \'etale and $Z\to X_2$ is identified with the zero section $Z\to Z\times \A^n$ for some integer $n\geq 0$.
The induced square
\[
\begin{tikzcd}
X_1-Z\ar[d]\ar[r]&
(X_1,Z)\ar[d]
\\
X-Z\ar[r]&
(X,Z)
\end{tikzcd}
\]
is a strict Nisnevich distinguished square.
Hence the induced square
\[
\begin{tikzcd}
\Th_{(X,Z)}^{-1}(u^* \Omega_{X/S}^1)\ar[d]\ar[r]&
\Th_{(X_1,Z)}^{-1}(u_1^* \Omega_{X_1/S}^1)\ar[d]
\\
\Th_{X-Z}^{-1}(\Omega_{X-Z/S}^1)\ar[r]&
\Th_{X_1-Z}^{-1}(\Omega_{X_1-Z/S}^1)
\end{tikzcd}
\]
in $\sT(S)$ is cartesian,
where $u_1\colon (X_1,Z)\to X_1$ is the canonical morphism.
Similarly,
the induced square
\[
\begin{tikzcd}
\Th_{(X,Z)}^{-1}(\Omega_{(X,Z)/S}^1)\ar[d]\ar[r]&
\Th_{(X_1,Z)}^{-1}(\Omega_{(X_1,Z)/S}^1)\ar[d]
\\
\Th_{X-Z}^{-1}(\Omega_{X-Z/S}^1)\ar[r]&
\Th_{X_1-Z}^{-1}(\Omega_{X_1-Z/S}^1)
\end{tikzcd}
\]
is cartesian.
Hence \eqref{blow.4.1} is equivalent to
\[
\Th_{(X_1,Z)}^{-1}(u_1^* \Omega_{X/S}^1)
\xrightarrow{\simeq}
\Th_{(X_1,Z)}^{-1}(\Omega_{(X_1,Z)/S}^1)
\]
in $\sT(S)$.
By repeating this argument for $X_1\to X_2$,
we reduce to the case of $(X_2,Z)$,
so we may assume that $Z\to X$ is identified with the zero section $Z\to Z\times \A^n$.

With this reduction,
we can write $X=Z[x_1,\ldots,x_n]$.
Then $\Bl_Z X$ is isomorphic to the gluing of
\[
U_i:=Z[x_i,x_1/x_i,\ldots,x_n/x_i]
\]
for all $1\leq i\leq n$.
Also,
$u_{U_i}^*\Omega_{X/Z}^1$ has a basis $\{dx_1,\ldots,dx_n\}$,
and $\Omega_{(U_i,U_i\times_X Z)/Z}^1$ has a basis $\{(dx_1)/x_i,\ldots,(dx_n)/x_i\}$,
where $u_{U_i}\colon (U_i,U_i\times_X Z)\to X$ is the canonical morphism.
Hence we have an isomorphism
\begin{equation}
\label{blow.5.1}
u_{U_i}^*\Omega_{X/Z}^1
\simeq
\Omega_{(U_i,U_i\times_X Z)/Z}^1(-E_{U_i}),
\end{equation}
where $E_{U_i}$ is the restriction of the exceptional divisor on $\Bl_Z X$ to $U$.

By Zariski descent,
it suffices to show that the induced morphism
\[
\Th_{(U,U\times_X Z)}(\Omega_{(U,U\times_X Z)/S}^1)
\to
\Th_{(U,U\times_X Z)}(u_U^*\Omega_{X/S}^1)
\]
in $\sT(U,U\times_X Z)$ is an isomorphism for every intersection $U$ of $U_i$'s.
Here, let $u_U\colon (U,U\times_X Z)\to X$ be the canonical morphism.
We have the compatible decompositions
\[
\Omega_{X/S}^1
\simeq
\Omega_{X/Z}^1\oplus \Omega_{Z/S}^1,
\;
\Omega_{(U,U\times_X Z)/S}^1
\simeq
\Omega_{(U,U\times_X Z)/Z}^1\oplus \Omega_{Z/S}^1.
\]
Hence by Proposition \ref{Gysin.9},
it suffices to show that the induced morphism
\[
\Th_{(U,U\times_X Z)}(\Omega_{(U,U\times_X Z)/Z}^1)
\to
\Th_{(U,U\times_X Z)}(u_U^*\Omega_{X/Z}^1)
\]
in $\sT(U,U\times_X Z)$
is an isomorphism.
Due to \eqref{blow.5.1},
this is a consequence of Proposition \ref{blow.5}.
\end{proof}

Now, we discuss the evaluation morphism allowing a nontrivial log structure as follows:

\begin{const}
\label{blow.3}
Let $S\in \lSch$,
and let $Z\to X$ be a closed immersion in $\lSm_S^\sat$.
Regard $(X,Z)$ as a strict closed subscheme of $X\times_S (X,Z)$ via the graph morphism $(X,Z)\to X\times_S (X,Z)$.
Consider the canonical morphism $j\colon (X,Z)\to X$.
Using Construction \ref{Gysin.10},
we obtain a log Gysin sequence
\[
\Th_{(X\times_S (X,Z),(X,Z))}^{-1}(v^*\Omega_{X/S}^1)
\to
\Th_{X\times_S (X,Z)}^{-1}(u^*\Omega_{X/S}^1)
\to
M_S(X,Z),
\]
where $u$ is the composite of the second projection $X\times_S (X,Z)\to (X,Z)$ and the canonical morphism $(X,Z)\to X$,
and $v$ is induced by $u$.

We have an induced strict closed immersion
\begin{equation}
\label{blow.3.2}
Z\times_S (X,Z)\to (X\times_S (X,Z),(X,Z)),
\end{equation}
which is a ``log compactification'' of
\[
Z\times_X (X-Z) \to (X\times_S (X-Z))-(X-Z).
\]
Hence we obtain a natural null sequence
\begin{equation}
\label{blow.3.1}
\Th_{Z\times_S (X,Z)}^{-1}(w^*\Omega_{X/S}^1)
\to
\Th_{X\times_S (X,Z)}^{-1}(u^*\Omega_{X/S}^1)
\to
M_S(X,Z),
\end{equation}
where $w\colon Z\times_S (X,Z)\to X$ is induced by $u$.

Let $j\colon (X,Z)\to X$ be the canonical morphism.
Now, we have the morphism
\[
\ev\colon M_S(X/Z)\otimes_S \Th_{(X,Z)}^{-1}(\Omega_{(X,Z)/S}^1)\to \unit_S
\]
given by the composite
\begin{align*}
&
M_S(X/Z) \otimes_S \Th_{(X,Z)}^{-1}(\Omega_{(X,Z)/S}^1)
\\
\xrightarrow{\simeq}
&
M_S(X/Z)
\otimes_S
\Th_{(X,Z)}^{-1}(j^*\Omega_{X/S}^1)
\\
\xrightarrow{\simeq} &
\cofib(\Th_{Z\times_S (X,Z)}^{-1}(w^*\Omega_{X/S}^1)
\to
\Th_{X\times_S (X,Z)}^{-1}(u^*\Omega_{X/S}^1))
\\
\xrightarrow{\eqref{blow.3.1}}&
M_S(X,Z)
\to \unit_S,
\end{align*}
where the first morphism is obtained by Proposition \ref{blow.4},
and the last morphism is induced by the structural morphism $(X,Z)\to S$.
By adjunction,
we obtain a morphism
\[
\comp
\colon
\Th_{(X,Z)}^{-1}(\Omega_{(X,Z)/S}^1)
\to
D(X/Z).
\]
\end{const}

The following is the key technical result in this paper:

\begin{prop}
\label{Gysin.13}
Let $S\in \lSch$,
and let $i\colon Z\to X$ be a closed immersion in $\lSm_S^\sat$.
Then there exists a morphism of cofiber sequences
\[
\begin{tikzcd}
\Th_{(X,Z)}^{-1}(\Omega_{(X,Z)/S}^1)\ar[r,"\Gys"]\ar[d,"\comp"']&
\Th_Z^{-1}(\Omega_{X/S}^1)\ar[r]\ar[d,"\comp"]&
\Th_X^{-1}(\Omega_{Z/S}^1)\ar[d,"\comp"]
\\
D(M_S(X/Z))
\ar[r]&
D(M_S(X))\ar[r]&
D(M_S(Z)).
\end{tikzcd}
\]
\end{prop}
\begin{proof}
Let $j\colon (X,Z)\to X$ be the canonical morphism.
For abbreviations,
we set
\[
A:=\Th_Z^{-1}(\Omega_{Z/S}^1),
\text{ }
B:\Th_X^{-1}(\Omega_{X/S}^1),
\text{ }
C:=\Th_{(X,Z)}^{-1}(j^* \Omega_{X/S}^1).
\]
Observe that we have $C\simeq \Th_{(X,Z)}^{-1}(\Omega_{(X,Z)/S}^1)$ by Proposition \ref{blow.4}.
Regard a morphism of cofiber sequences as a morphism of commutative squares.
By adjunction,
it suffices to construct a commutative cube
\begin{equation}
\label{Gysin.13.1}
\begin{tikzcd}[column sep=tiny, row sep=tiny]
&
M_S(Z)\otimes_S C\ar[ld]\ar[rr]\ar[dd]&
&
0\ar[ld]\ar[dd]
\\
M_S(X)\otimes_S C\ar[rr,crossing over]\ar[dd]&
&
M_S(X/Z) \otimes_S C
\\
&
M_S(Z)\otimes_S B\ar[rr,"\Gys",near start]\ar[ld]&
&
M_S(Z)\otimes_S A\ar[ld,"\ev"]
\\
M_S(X)\otimes_S B\ar[rr,"\ev"']&
&
1_S.\ar[uu,crossing over,leftarrow,"\ev" very near start]
\end{tikzcd}
\end{equation}
Using \eqref{blow.3.2},
we obtain a commutative cube.
\[
\begin{tikzcd}[column sep=-0.1em, row sep=tiny]
&
M_S(Z\times_S (X,Z))\ar[ld]\ar[rr]\ar[dd]&
&
0\ar[ld]\ar[dd]
\\
M_S(X\times_S (X,Z))\ar[rr,crossing over]\ar[dd]&
&
M_S\left(\frac{X\times_S (X,Z)}{(X\times_S (X,Z),(X,Z))}\right)
\\
&
M_S(Z\times_S X)\ar[rr]\ar[ld]&
&
M_S\left(\frac{Z\times_S X}{(Z\times_S X,Z)}\right)\ar[ld]
\\
M_S(X\times_S X)\ar[rr]&
&
M_S\left(\frac{X\times_S X}{(X\times_S X,X)}\right),\ar[uu,crossing over,leftarrow]
\end{tikzcd}
\]
where $X$ in $(X\times_S X,X)$ is regarded as the diagonal strict closed subscheme of $X\times_S X$.
We also have natural isomorphisms
\begin{gather*}
M_S\left(\frac{X\times_S (X,Z)}{(X\times_S (X,Z),(X,Z))}\right)
\simeq
\Th_{(X,Z)}(u^*\Omega_{X/S}^1),
\\
M_S\left(\frac{Z\times_S X}{(Z\times_S X,Z)}\right)
\simeq
\Th_X(\Omega_{X/S}^1),
\;
M_S\left(\frac{Z\times_S X}{(Z\times_S X,Z)}\right)
\simeq
\Th_Z(i^*\Omega_{X/S}^1),
\end{gather*}
where $u\colon (X,Z)\to X$ is the canonical morphisms,
and we need Propositions \ref{Gysin.27} and \ref{Gysin.28} for the first one.
Together with the usual method in Theorem \ref{setup.1} and Construction \ref{Gysin.10},
we obtain a commutative cube
\begin{equation}
\label{Gysin.13.2}
\begin{tikzcd}[column sep=tiny, row sep=tiny]
&
M_S(Z)\otimes_S C\ar[ld]\ar[rr]\ar[dd]&
&
0\ar[ld]\ar[dd]
\\
M_S(X)\otimes_S C\ar[rr,crossing over,"\Gys",near end]\ar[dd]&
&
M_S(X,Z)
\\
&
M_S(Z)\otimes_S B\ar[rr,"\Gys",near start]\ar[ld]&
&
M_S(Z)\ar[ld]
\\
M_S(X)\otimes_S B\ar[rr,"\Gys"']&
&
M_S(X).\ar[uu,crossing over,leftarrow]
\end{tikzcd}
\end{equation}
By Theorem \ref{Gysin.12},
we have a commutative prism
\begin{equation}
\label{Gysin.13.3}
\begin{tikzcd}[row sep=tiny]
&
&
&
0\ar[dd]\ar[ld]
\\
M_S(Z)\otimes_S C\ar[rrru,bend left=10]\ar[dd]\ar[rr] &
&
0
\\
&
&
&
M_S(Z)\otimes_S A\ar[ld,"\Gys"]
\\
M_S(Z)\otimes_S B\ar[rr,"\Gys"']\ar[rrru,bend left=10,"\Gys"]&
&
M_S(Z).\ar[uu,leftarrow,crossing over]
\end{tikzcd}
\end{equation}
Construction \ref{blow.3} yields a commutative prism
\begin{equation}
\label{Gysin.13.4}
\begin{tikzcd}[column sep=tiny, row sep=tiny]
&
&
&
0\ar[dd]\ar[ld]
\\
&
&
M_S(X/Z)\otimes_S C
\\
&
M_S(Z)\otimes_S C\ar[ld]\ar[rr]\ar[rruu,bend left=25]&
&
0\ar[ld]
\\
M_S(X)\otimes_S C\ar[rr,"\Gys"']\ar[rruu,bend left=25,crossing over]&
&
M_S(X,Z).\ar[uu,crossing over,leftarrow]
\end{tikzcd}
\end{equation}
We obviously have a commutative triangle
\begin{equation}
\label{Gysin.13.5}
\begin{tikzcd}[column sep=tiny, row sep=small]
&
M_S(Z)\ar[ld]\ar[rd]
\\
M_S(X)\ar[rr]&
&
\unit_S.
\end{tikzcd}
\end{equation}
Assemble \eqref{Gysin.13.2}--\eqref{Gysin.13.5} to obtain \eqref{Gysin.13.1}.
\end{proof}

In the special case of $X=\P^n$ and $Z=\P^{n-1}$,
we have the following result:

\begin{prop}
\label{Gysin.14}
Let $S\in \Sch$.
Then
\[
\ev\colon
M_S(\P^n/\P^{n-1})
\otimes_S
\Th_{(\P^n,\P^{n-1})}^{-1}(\Omega_{(\P^n,\P^{n-1})/S}^1)
\to
\unit_S
\]
is an isomorphism for every integer $n\geq 1$.
\end{prop}
\begin{proof}
Consider the inclusion $i\colon S\to (\P^n,\P^{n-1})$ at $o:=[0:\cdots:0:1]\in \P^n$.
By Lemma \ref{Gysin.15} below,
it suffices to show that the composite morphism
\begin{align*}
 &M_S(\P^n/\P^{n-1})
\otimes_S
\Th_S^{-1}(\cO^n)
\\
\to &
M_S(\P^n/\P^{n-1})
\otimes_S
\Th_{(\P^n,\P^{n-1})}^{-1}(\Omega_{(\P^n,\P^{n-1})/S}^1)
\to
\unit_S
\end{align*}
is an isomorphism.
This is homotopic to the composite
\begin{align*}
& M_S(\P^n/\P^{n-1})
\otimes_S
\Th_{S}^{-1}(\cO^n)
\\
\xrightarrow{\simeq} &
\cofib(
\Th_{\P^{n-1}}^{-1}(\cO^n)
\to
\Th_{\P^n}^{-1}(\cO^n))
\to
\unit_S,
\end{align*}
where the second morphism is obtained by the log Gysin sequence
\[
\Th_{(\P^n,\{o\})}^{-1}(\cO^n) \to\Th_{\P^n}^{-1}(\cO^n ) \to \unit_S
\]
and the morphism
\begin{equation}
\label{Gysin.14.1}
\Th_{\P^{n-1}}^{-1}(\cO^n ) \to\Th_{(\P^n,\{o\})}^{-1}(\cO^n)
\end{equation}
induced by the strict closed immersion $\P^{n-1}\to (\P^n,\{o\})$.
We conclude since \eqref{Gysin.14.1} is an isomorphism by \cite[Lemma 7.1.9, Proposition 7.1.10]{logDM}.
\end{proof}

\begin{lem}
\label{Gysin.15}
Let $S\in \Sch$.
Then the morphism
\[
\Th_S^{-1}(\cO)
\to
\Th_{(\P^n,\P^{n-1})}^{-1}(\Omega_{(\P^n,\P^{n-1})/S}^1)
\]
induced by the closed immersion $S\to (\P^n,\P^{n-1})$ at $o:=[0:\cdots:0:1]\in \P^n$ is an isomorphism.
\end{lem}
\begin{proof}
Identify $\Omega_{(\P^n,\P^{n-1})/S}^1$ with $\cO^n(-D)$,
where $D:=\P^{n-1}$.
By the $(\P^n,\P^{n-1})$-invariance \cite[Proposition 7.3.1]{logDM},
we have an isomorphism
\[
\Th_{(\P^n,\P^{n-1})}^{-1}(\cO^n)
\simeq
\Th_S^{-1}(\cO^n)
\]
in $\sT(S)$.
Hence it suffices to show that the morphism
\[
\Th_{(\P^n,\P^{n-1})}(\cO^n(-D))
\to
\Th_{(\P^n,\P^{n-1})}(\cO^n)
\]
in $\sT(\P^n,\P^{n-1})$ induced by the canonical inclusion $\cO(-D)\to \cO$ is an isomorphism.
This is a consequence of Proposition \ref{blow.5}.
\end{proof}

Let us discuss two results on dualizable objects and $\mathfrak{p}_f$.

\begin{prop}
\label{Gysin.22}
Let $f\colon X\to S$ be a strict smooth separated morphism of fs log schemes,
and let $\cF\in \sT(X)$ be a dualizable object.
Assume that $\mathfrak{p}_f$ is an isomorphism.
Then $f_\sharp \cF$ is dualizable with dual $f_* D\cF$.
\end{prop}
\begin{proof}
We have natural isomorphisms
\[
\ul{\Hom}_S(f_\sharp \cF,-)
\simeq
f_*(D\cF\otimes_X f^*(-))
\simeq
f_*D \cF \otimes_S (-),
\]
where the second isomorphism is due to Proposition \ref{Gysin.23}.
Whence the result follows.
\end{proof}

\begin{prop}
\label{alpha.3}
Let $f\colon X\to S$ be a strict smooth separated morphism of fs log schemes.
Then $\mathfrak{p}_f$ is an isomorphism if and only if the following conditions are satisfied:
\begin{enumerate}
\item[\textup{(i)}]
$M_S(X)$ is dualizable.
\item[\textup{(ii)}]
The comparison morphism
\[
\comp\colon \Th_X^{-1}(\Omega_{X/S}^1)
\to
D(M_S(X))
\]
is an isomorphism.
\end{enumerate}
\end{prop}
\begin{proof}
Repeat \cite[Proof of Proposition 5.7]{AHI} (see also \cite[Proof of Proposition 1.7.16]{Ayo071} and \cite[Proof of Theorem 2.4.42]{CD19}) to show that under assuming (i), $\mathfrak{p}_f$ is an isomorphism if and only if (ii) is satisfied.
To conclude,
observe that (i) is satisfied if $\mathfrak{p}_f$ is an isomorphism by Proposition \ref{Gysin.22}.
\end{proof}

Now, we are ready to prove the main result of this section.

\begin{prop}
\label{Gysin.21}
Let $S\in \lSch$,
and let $p\colon \P^n\to S$ be the projection with an integer $n\geq 0$.
Then $\mathfrak{p}_p$ is an isomorphism.
\end{prop}
\begin{proof}
By Proposition \ref{alpha.3},
it suffices to show that $M_S(\P^n)$ is dualizable and the comparison morphism
\[
\comp\colon \Th_{\P^n}^{-1}(\Omega_{\P^n/S}^1)
\to
D(M_S(\P^n))
\]
is an isomorphism.
Consider the morphism $q\colon S\to \ul{S}$ removing the log structure.
Since $q^*$ is monoidal,
it suffices to show the claim for $\ul{S}$,
so we may assume $S\in \Sch$.

We proceed by induction on $n$.
The claim is obvious if $n=0$.
Assume $n>0$.
Note that $M_S(\P^n/\P^{n-1})\in\sT(S)$ is dualizable since it is isomorphic to $\unit_S(n)[2n]$ by \cite[Proposition 3.2.7]{logSH}.
Apply Proposition \ref{Gysin.13} to $\P^{n-1}\to \P^n$.
By induction,
it suffices to show that the comparison morphism
\[
\comp
\colon
\Th_{(\P^n,\P^{n-1})}^{-1}(\Omega_{(\P^n,\P^{n-1})/S}^1)
\to
D(M_S(\P^n/\P^{n-1}))
\]
is an isomorphism.
This is a consequence of Proposition \ref{Gysin.14}.
\end{proof}

\section{Poincar\'e duality}
\label{purity}

We keep working with $\sT$ in \S \ref{sat}.
The following is a non $\A^1$-invariant analogue of \cite[(1.36)]{Ayo071}.

\begin{thm}
\label{local.1}
Let $S\in \lSch$,
and 
let $i\colon Z\to X$ be a strict closed immersion in $\lSm_S^\sat$,
and let $f\colon X\to S$ be the structural morphism.
Then the composite natural transformation
\[
f_\sharp i_*
\xrightarrow{\Gys}
(fi)_\sharp \Sigma^{\rN_i}i^*i_*
\xrightarrow{\ad'}
(fi)_\sharp \Sigma^{\rN_i}
\]
is an isomorphism.
\end{thm}
\begin{proof}
Let $v\colon V\to X$ be a strict smooth morphism of fs log schemes,
and form a cartesian square
\[
\begin{tikzcd}
W\ar[r,"b"]\ar[d,"w"']&
V\ar[d,"v"]
\\
Z\ar[r,"i"]&
X.
\end{tikzcd}
\]
Let $\beta_V$ (resp.\ $\gamma_V$) be the composite of  morphisms in the lower (resp.\ upper) row of
\[
\begin{tikzcd}
(fv)_\sharp v^* i_*\ar[r,"\Gys"]\ar[d,"\Ex"',"\simeq"]&
(fvb)_\sharp \Sigma^{\rN_b} b^*v^*i_*\ar[r,"\simeq"]\ar[d,"\Ex"]&
(fiw)_\sharp w^* \Sigma^{\rN_i} i^* i_*\ar[r,"\ad'"]&
(fiw)_\sharp w^*\Sigma^{\rN_i} \ar[d,"\simeq"]
\\
(fv)_\sharp b_* w^*\ar[r,"\Gys"]&
(fvb)_\sharp \Sigma^{\rN_b} b^*b_* w^*\ar[rr,"\ad'"]&
&
(fvb)_\sharp \Sigma^{\rN_b} w^*.
\end{tikzcd}
\]
Then
\begin{equation}
\label{local.1.1}
\text{$\beta_V$ is an isomorphism if and only if $\gamma_V$ is an isomorphism.}
\end{equation}

We have an induced commutative square
\[
\begin{tikzcd}
M_S(V)\ar[r,"\Gys"]\ar[d]&
\Th_W(\rN_a)\ar[d]
\\
M_S(X)\ar[r,"\Gys"]&
\Th_Z(\rN_i).
\end{tikzcd}
\]
Argue as in Theorem \ref{setup.1} to upgrade this square to a commutative diagram
\[
\begin{tikzcd}
(fv)_\sharp v^*\ar[d,"\ad'"']\ar[r,"\Gys"]&
(fvb)_\sharp \Sigma^{\rN_b} b^*v^*\ar[r,"\simeq"]&
(fiw)_\sharp w^*\Sigma^{\rN_i} i^*\ar[d,"\ad'"]
\\
f_\sharp\ar[rr,"\Gys"]&
&
(fi)_\sharp \Sigma^{\rN_i} i^*.
\end{tikzcd}
\]
Hence the diagram
\[
\begin{tikzcd}
(fv)_\sharp v^*i_*\ar[d,"\ad'"']\ar[r,"\Gys"]&
(fvb)_\sharp \Sigma^{\rN_b} b^*v^*i_*\ar[r,"\simeq"]&
(fiw)_\sharp w^*\Sigma^{\rN_i} i^*i_*\ar[d,"\ad'"]\ar[r,"\ad'"]&
(fiw)_\sharp w^*\Sigma^{\rN_i}\ar[d,"\ad'"]
\\
f_\sharp i_*\ar[rr,"\Gys"]&
&
(fi)_\sharp \Sigma^{\rN_i} i^*i_*\ar[r,"\ad'"]&
(fi)_\sharp \Sigma^{\rN_i}
\end{tikzcd}
\]
commutes,
i.e.,
\[
\begin{tikzcd}
(fv)_\sharp v^*i_*\ar[r,"\gamma_V"]\ar[d,"\ad'"']&
(fiw)_\sharp w^*\Sigma^{\rN_i} \ar[d,"\ad'"]
\\
f_\sharp i_*\ar[r,"\gamma_X"]&
(fi)_\sharp \Sigma^{\rN_i}
\end{tikzcd}
\]
commutes.
Now,
let
\[
\begin{tikzcd}
V''\ar[d]\ar[r]&
V\ar[d,"v"]
\\
V'\ar[r,"v'"]&
X
\end{tikzcd}
\]
be a strict Nisnevich distinguished square,
and consider the pullback
\[
\begin{tikzcd}
W''\ar[d]\ar[r]&
W\ar[d,"w"]
\\
W'\ar[r,"w'"]&
Z.
\end{tikzcd}
\]
Let $v''\colon V''\to X$ and $w''\colon W''\to X$ be the composite morphisms.
Arguing as above,
we obtain a commutative cube
\[
\begin{tikzcd}[column sep=tiny, row sep=tiny]
&
(fv'')_\sharp v''^* i_*\ar[ld]\ar[rr]\ar[dd,"\gamma_{V''}",near start]&
&
(fv)_\sharp v^* i_*\ar[ld]\ar[dd,"\gamma_V"]
\\
(fv')_\sharp v'^* i_*\ar[rr,crossing over]\ar[dd,"\gamma_{V'}"']&
&
f_\sharp  i_*\ar[dd,"\gamma_X",near end]
\\
&
(fiw'')_\sharp w''^*\Sigma^{\rN_i}\ar[rr]\ar[ld]&
&
(fiw)_\sharp w^*\Sigma^{\rN_i}\ar[ld]
\\
(fiw')_\sharp w'^*\Sigma^{\rN_i}\ar[rr]&
&
(fi)_\sharp \Sigma^{\rN_i}\ar[uu,crossing over,leftarrow]
\end{tikzcd}
\]
whose top and bottom squares are cartesian by strict Nisnevich descent.
Together with \eqref{local.1.1},
we deduce the following claim:
If three of $\beta_X$, $\beta_V$, $\beta_{V'}$, and $\beta_{V''}$ are isomorphisms,
then the remaining one is an isomorphism too.
In particular, the question is strict Nisnevich local on $X$.

Hence using
\cite[Proposition C.1]{divspc},
as in the proof of Theorem \ref{setup.1},
we reduce to the case where $f$ is strict smooth.
By \cite[Proof of Lemma 2.28 in \S 3]{MV},
we may also assume that there exists a commutative diagram with cartesian squares
\[
\begin{tikzcd}
Z\ar[d]\ar[r,"\id",leftarrow]&
Z\ar[r,"\id"]\ar[d]&
Z\ar[d]
\\
X\ar[r,leftarrow,"u"]&
X'\ar[r,"u'"]&
X''
\end{tikzcd}
\]
with strict \'etale $u$ and $u'$ such that $Z\to X''$ is identified with the zero section $Z\to Z\times \A^s$ for some integer $s\geq 0$.
Using the Nisnevich coverings $\{X-Z,X'\to X\}$ and $\{X''-Z,X'\to X''\}$ and the above paragraph,
we reduce to the case of $Z\to Z\times \A^s$.
Using the Nisnevich covering $\{Z\times \A^s,\Z\times (\P^s-0)\to \Z\times \P^s\}$ and the above paragraph,
we also reduce to the case of $Z\to Z\times \P^s$.
Replacing $S$ by $Z$,
we may assume $S=Z$.

Consider the induced commutative diagram with cartesian squares
\[
\begin{tikzcd}
S\ar[r,"i"]\ar[d,"i"']&
X\ar[r,"f"]\ar[d,"i'"]&
S\ar[d,"i"]
\\
X\ar[r,"a"]&
X\times_S X\ar[r,"p_2"]&
X,
\end{tikzcd}
\]
where $a$ is the diagonal morphism, and $p_2$ is the second projection.
Let $p_1\colon X\times_S X\to X$ be the first projection.
The composite natural transformation
\[
\mathfrak{p}_f
\colon
f_\sharp
\xrightarrow{\simeq}
f_\sharp p_{1*} a_*
\xrightarrow{\Ex}
f_* p_{2\sharp} a_*
\xrightarrow{\Gys}
f_* \Sigma^{\Omega_f^1} a^*a_*
\xrightarrow{\ad'}
f_* \Sigma^{\Omega_f^1}
\]
is an isomorphism by Proposition \ref{Gysin.21}.
Since the second morphism is an isomorphism by Propositions \ref{Thom.2} and \ref{Gysin.21},
the composite natural transformation
\[
f_* p_{2\sharp} a_*
\xrightarrow{\Gys}
f_* \Sigma^{\Omega_f^1} a^*a_*
\xrightarrow{\ad'}
f_* \Sigma^{\Omega_f^1}
\]
is an isomorphism.
We also have a commutative diagram
\[
\begin{tikzcd}[row sep=small]
f_* p_{2\sharp} a_* i_*\ar[r,"\Gys"]\ar[d,"\simeq"']&
f_* \Sigma^{\Omega_f^1} a^*a_* i_*\ar[r,"\ad'"]\ar[d,"\simeq"]&
f_* \Sigma^{\Omega_f^1} i_*\ar[r,"\simeq"]&
f_* i_* \Sigma^{\rN_i}\ar[d,"\simeq"]
\\
f_*p_{2\sharp}i_*'i_*\ar[d,"\Ex"']\ar[r,"\Gys"]\ar[rrd,"(b)",phantom]&
f_*\Sigma^{\Omega_f^1}a^*i_*'i_*\ar[r,"\Ex"]&
f_*\Sigma^{\Omega_f^1}i_*i^*i_*\ar[d,"\simeq"]\ar[r,"\ad'"]&
f_*\Sigma^{\Omega_f^1}i_*\ar[d,"\simeq"]
\\
f_*i_*f_\sharp i_*\ar[rr,"\Gys",near start]\ar[d,"\simeq"']&
&
f_*i_*\Sigma^{\rN_i} i^*i_*\ar[r,"\ad'"]\ar[d,"\simeq"]&
f_*i_*\Sigma^{\rN_i}\ar[d,"\simeq"]
\\
f_\sharp i_*\ar[rr,"\Gys"]&
&
\Sigma^{\rN_i} i^*i_*\ar[r,"\ad'"]&
\Sigma^{\rN_i},
\end{tikzcd}
\]
where $(b)$ is obtained from \eqref{Thom.1.2}.
To conclude,
observe that the natural transformation $\Ex\colon p_{2\sharp}u_* \to i_*f_\sharp$ is an isomorphism by Propositions \ref{Thom.2} and \ref{Gysin.21}.
\end{proof}

Now, we recover the motivic ambidexterity of Annala-Hoyois-Iwasa \cite[Theorem 1.1(i)]{AHI} in the log setting:

\begin{thm}
\label{local.2}
Let $f\colon X\to S$ be a strict smooth projective morphism of fs log schemes.
Then $\mathfrak{p}_f$ is an isomorphism.
\end{thm}
\begin{proof}
By Proposition \ref{Thom.2},
we can work Zariski locally on $S$,
so we may assume that $f$ admits a factorization $X\xrightarrow{i} \P^n\xrightarrow{p}S$ such that $i$ is a closed immersion.
By Proposition \ref{Thom.3} and Theorem \ref{local.1},
$\mathfrak{p}_f$ is an isomorphism if $\mathfrak{p}_p$ is an isomorphism.
Proposition \ref{Gysin.21} finishes the proof.
\end{proof}

\begin{thm}
\label{local.3}
Let $S\in \lSch$,
and let $i\colon Z\to X$ be a strict closed immersion of strict smooth projective fs log schemes over $S$.
Then the comparison morphism
\[
\comp
\colon
\Th_{(X,Z)}^{-1}(\Omega_{(X,Z)/S}^1)\to D(M_S(X/Z))
\]
is an isomorphism.
\end{thm}
\begin{proof}
This is an immediate consequence of Propositions \ref{Gysin.13} and \ref{alpha.3} and Theorem \ref{local.2}.
\end{proof}

\begin{rmk}
Let $S\in \Sch$,
and let
\[
\sT\in \Sh_{\sNis}(\SmlSm_S,\CAlg(\Prst))
\]
(instead of $\sT\in \Sh_\sNis(\lSch,\CAlg(\Prst))$) be a strict Nisnevich sheaf of presentably symmetric monoidal stable $\infty$-categories satisfying the conditions in \S \ref{sat}.
Then Theorem \ref{local.2} holds for this $\sT$ and every smooth projective morphism of schemes $f\colon X\to S$ by the same proof.
This applies to Theorem \ref{blow.6} below also.
\end{rmk}

\section{Poincar\'e duality for blow-up squares}
\label{blow}

We keep working with $\sT$ in \S \ref{sat}.
We will show in Theorem \ref{blow.6} that Poincar\'e duality holds for three corners of a smooth blow-up square if and only if Poincar\'e duality holds for the remaining corner too.

\begin{const}
Let $S\in \Sch$, let $X\in \Sm_S$,
and let $Y$ and $Z$ be transversal strict closed subschemes of $X$ with $Y,Z\in \Sm_S$.
We set $W:=Y\times_X Z$ for simplicity of notation.
Then we have induced morphisms
\begin{equation}
\label{blow.7.1}
\Bl_{\Bl_W Y} (\Bl_Z X)
\leftarrow
\Bl_{\Bl_W Y} (\Bl_{\Bl_W Z}(\Bl_W X))
\to
\Bl_{\Bl_W Y} (\Bl_W X)
\to
\Bl_Y X.
\end{equation}
\end{const}

\begin{prop}
\label{blow.2}
Let $S\in \Sch$, and let $i\colon Z\to X$ be a closed immersion in $\Sm_S$.
Then there exists a commutative square
\[
\begin{tikzcd}
\Th_{(X,Z)}^{-1}(\Omega_{(X,Z)/S}^1)\ar[r,"\comp"]\ar[d,"\simeq"']&
D(M_S(X/Z))\ar[d,"\simeq"]
\\
\Th_{(X',Z')}^{-1}(\Omega_{(X',Z')/S}^1)\ar[r,"\comp"]
&
D(M_S(X'/Z'))
\end{tikzcd}
\]
with vertical isomorphisms,
where $X':=\Bl_Z X$ and $Z':=Z\times_X \Bl_Z X$. 
\end{prop}
\begin{proof}
We set $Y:=(X,Z)=(X',Z')$.
Note that the right vertical morphism is an isomorphism by the smooth blow-up descent \cite[Theorem 3.3.7]{logSH}.
The left vertical morphism is the identity after the identification $(X,Z)=(X',Z')$.

We have the graph morphism $\Gamma\colon Y\to X\times_S Y$,
and we often regard $Y$ as a strict closed subscheme of $X\times_S Y$ by this way.
Apply \eqref{blow.7.1} to the transversal smooth subschemes $\ul{Y}\times 0$ and $Z\times_S \ul{Y}\times 0$ of $X\times_S \ul{Y}\times \P^1$ to construct morphisms
\[
\Bl_{\ul{Y}}(X'\times_S \ul{Y} \times \P^1)
\leftarrow
\ul{V}
\to
\Bl_{\ul{Y}}(X\times_S \ul{Y}\times \P^1),
\]
where $\ul{V}$ corresponds to the second term in \eqref{blow.7.1}.
Since we have a natural morphism
\[
\Bl_Y(X'\times_S Y\times \square) - \partial \Bl_Y(X'\times_S Y\times \square)
\to
\Bl_Y(X\times_S Y\times \square) - \partial \Bl_Y(X\times_S Y\times \square),
\]
we have induced morphisms
\[
\Bl_Y(X'\times_S Y\times \square)
\leftarrow V
\to \Bl_Y(X\times_S Y\times \square)
\]
for some $V$ with the underlying scheme $\ul{V}$ such that the left morphism is an admissible blow-up along a smooth center.
By restricting this,
we also obtain
\[
\rD_Y (X'\times_S Y)\leftarrow W 
\to
\rD_Y(X\times_S Y)
\]
for some $W$ such that the left morphism is an admissible blow-up too.
This yields morphisms of diagrams
\begin{equation}
\label{blow.2.2}
(\eqref{Gysin.0.2}\text{ for }Y\to X'\times_S Y)
\leftarrow P
\to
(\eqref{Gysin.0.2}\text{ for }Y\to X\times_S Y)
\end{equation}
for some $P$ such that the left morphism consists of admissible blow-ups along smooth centers.
Using the invariance under admissible blow-ups along smooth centers \cite[Theorem 7.2.10]{logDM},
we obtain a morphism of cofiber sequences
\begin{equation}
\label{blow.2.1}
\begin{tikzcd}
M_S(X'\times_S Y,Y)\ar[d]\ar[r]&
M_S(X'\times_S Y)\ar[r]\ar[d]&
\Th_Y(\rN_Y(X'\times_S Y))\ar[d]
\\
M_S(X\times_S Y,Y)\ar[r]&
M_S(X\times_S Y)\ar[r]&
\Th_Y(\rN_Y(X\times_S Y)).
\end{tikzcd}
\end{equation}

The $\rN$ part of \eqref{blow.2.2} can be written as
\[
\begin{tikzcd}[column sep=small]
(\rN_Y(X'\times_S Y),Y)
\ar[r,leftarrow]\ar[d]&
Q'\ar[r]\ar[d]&
(\rN_Y(X\times_S Y),Y)\ar[d]
\\
\rN_Y(X'\times_S Y)\ar[r,leftarrow]&
Q\ar[r]&
\rN_Y(X\times_S Y)
\end{tikzcd}
\]
for some $Q$ and $Q'$ such that the left horizontal morphisms are admissible blow-ups along smooth centers.
The right horizontal morphisms are composites of admissible blow-ups along smooth centers since
\begin{align*}
\rN_Y(X'\times_S Y)-\partial \rN_Y(X'\times_S Y)
\simeq &
\rN_{X-Z}(X\times_S (X-Z))
\\
\simeq & \rN_Y(X\times_S Y) - \partial \rN_Y(X\times_S Y).
\end{align*}
Hence the morphism
\[
\Th_Y(\rN_Y(X'\times_S Y))
\to
\Th_Y(\rN_Y(X\times_S Y))
\]
is an isomorphism.
We have an induced commutative square
\[
\begin{tikzcd}
Z'\times_S Y\ar[d]\ar[r]&
(X'\times_S Y,Y)\ar[d]
\\
Z\times_S Y\ar[r]&
(X\times_S Y,Y).
\end{tikzcd}
\]
Combine this with \eqref{blow.2.1} to obtain a morphism of null sequences
\[
\begin{tikzcd}
M_S(Z'\times_S Y)\ar[d]\ar[r]&
M_S(X'\times_S Y)\ar[r]\ar[d]&
\Th_Y(\rN_Y(X'\times_S Y))\ar[d,"\simeq"]
\\
M_S(Z\times_S Y)\ar[r]&
M_S(X\times_S Y)\ar[r]&
\Th_Y(\rN_Y(X\times_S Y)).
\end{tikzcd}
\]
Using the usual method as in Theorem \ref{setup.1} and Construction \ref{Gysin.10},
we can also obtain a morphism of null sequences
\[
\begin{tikzcd}
\Th_{Z'\times_S Y}^{-1}(w'^*\Omega_{X/S}^1)\ar[d]
\ar[r]&
\Th_{X'\times_S Y}^{-1}(u'^*\Omega_{X/S}^1)\ar[d]
\ar[r]&
M_S(Y)\ar[d,"="]
\\
\Th_{Z\times_S Y}^{-1}(w^*\Omega_{X/S}^1)
\ar[r]&
\Th_{X\times_S Y}^{-1}(u^*\Omega_{X/S}^1)
\ar[r]&
M_S(Y),
\end{tikzcd}
\]
where $u$ and $w$ are in Construction \ref{blow.3}, and $u'$ and $w'$ are induced by $u$ and $w$.
Together with Proposition \ref{blow.4},
we obtain a commutative triangle
\[
\begin{tikzcd}
M_S(X'/Z') \otimes_S\Th_Y^{-1}(\Omega_{(X',Z')/S}^1)
\ar[r,"\ev"]\ar[d,"\simeq"']&
\unit_S
\\
M_S(X/Z)\otimes_S \Th_Y^{-1}(\Omega_{(X,Z)/S}^1).\ar[ru,"\ev"']&
\end{tikzcd}
\]
By adjunction,
we obtain the desired commutative square.
\end{proof}

\begin{thm}
\label{blow.6}
Let $X\to S$ be a smooth morphism of schemes,
and let $Z\to X$ be a closed immersion in $\Sm_S$.
If $\mathfrak{p}$ is an isomorphism for three of 
\[
X,Z,\Bl_Z X,Z\times_X \Bl_Z X\to S,
\]
then $\mathfrak{p}$ is an isomorphism for the other too.
\end{thm}
\begin{proof}
This is an immediate consequence of Propositions \ref{Gysin.13} and \ref{blow.2}.
\end{proof}

For example, let $f\colon X\to S$ be a proper smooth morphism of schemes with morphisms
\[
X_n\to \cdots \to X_0  = X
\]
such that $X_{i+1}=\Bl_{Z_i} X_i$ for each $i$ with a smooth closed subscheme $Z_i$ of $X_i$.
If $X_n\to S$ is projective and $\mathfrak{p}$ is an isomorphism for $Z_i\to S$ for each $i$,
then $\mathfrak{p}_f$ is an isomorphism by Theorems \ref{local.2} and \ref{blow.6}.

By the way,
Theorem \ref{local.2} in the setting of $\sT=\logSH$ with $S=\Spec(k)$ for a perfect field $k$ admits an alternative proof (with a generalization to proper morphisms) assuming resolution of singularities:

\begin{thm}
Let $f\colon X\to k$ be a proper smooth morphism of schemes,
where $k$ is a perfect field admitting resolution of singularities.
Then $\mathfrak{p}_f$ for $\logSH$ is an isomorphism.
\end{thm}
\begin{proof}
By \cite[Theorem 2.4.50]{CD19},
$\mathfrak{p}_f$ for $\SH$ is an isomorphism.
Hence the conditions (i) and (ii) in Proposition \ref{alpha.3} are satisfied for $f$ and $\SH$.
Since $\omega^*$ is monoidal by \cite[Theorem 1.1(4)]{loghomotopy},
the conditions (i) and (ii) in Proposition \ref{alpha.3} are satisfied for $f$ and $\logSH$,
which implies the claim.
\end{proof}

\section{\texorpdfstring{$\omega^*\unit$}{w1}-modules}
\label{w1}

Annala-Hoyois-Iwasa \cite[\S 6]{AHI} considered modules over the $\A^1$-localized sphere spectrum.
Using this,
as a consequence of \cite[Theorem 1.1(i)]{AHI},
they constructed a cohomology theory on smooth schemes over a perfect field $k$ of characteristic $p>0$ that yields the cohomology of a log compactification.
The purpose of this section is to recover this result.

We work with an axiomatic setting as follows.
Let $\varphi_\sharp\colon \sT\to \sT'$ be a morphism in $\Sh_\sNis(\lSch,\CAlg(\Prst))$ such that $\sT$ and $\sT'$ satisfy the conditions in \S \ref{sat}, $\varphi_\sharp\colon \sT(S)\to \sT'(S)$ preserves colimits and is symmetric monoidal for every $S\in \lSch$,
the right adjoint $\varphi_*\colon \sT'(S)\to \sT(S)$ of $\varphi^*$ is fully faithful for every $S\in \lSch$,
and $\varphi_\sharp$ commutes with $f_\sharp$ for every $f\in \lSm^\sat$.

\begin{const}
\label{coh.6}
Consider the cartesian $\lSch$-family of symmetric monoidal $\infty$-categories \cite[Definition 4.8.3.1]{HA}
\[
\sT^\otimes \to 
\Fin_* \times \lSch.
\]
Let $\CAlg(\sT)$ be the $\infty$-category of pairs 
\[
(S\in \lSch,A\in \CAlg(\sT(S))).
\]
Using \cite[Lemma 4.8.3.13]{HA},
we obtain a functor
\[
\sT(-;-)
\colon
\CAlg(\sT)^\op
\to
\CAlg(\PrL)
\]
sending a pair $(S,A)$ to $\Mod_A(\sT(S))$, a morphism of the form $(f,c)\colon (X,f^* A)\to (S,A)$ (with the canonical $c$) to
\[
f^*\colon \Mod_A(\sT(S))\to \Mod_{f^*A}(\sT(X)),
\]
and a morphism of the form $(\id,g)\colon (S,A)\to (S,B)$ to
\[
-\otimes_A B\colon \Mod_A(\sT(S))\to \Mod_B(\sT(S)).
\]

The adjoint pair $(\varphi_\sharp,\varphi^*)$ yields the section
\[
\varphi^*\unit
\colon
\lSch\to \CAlg(\sT)
\]
of the cartesian fibration $\CAlg(\sT)\to \lSch$.
The composite $\sT(-;-)\circ \varphi^*\unit$ yields
\[
\sT(-;\varphi^*1)
\colon
\lSch^\op
\to
\CAlg(\PrL).
\]
This sends $S\in \lSch$ to the symmetric monoidal $\infty$-category of $\unit_S$-module objects of $\sT(X)$, and a morphism $X\to S$ in $\lSch$ to the functor
\begin{equation}
\label{coh.6.1}
f^*:=
f^*(-) \otimes_{f^*\varphi^*1_S} \varphi^* 1_X
\colon
\sT(S;\varphi^*1)
\to
\sT(X;\varphi^*1).
\end{equation}
Moreover,
we have an adjoint pair
\[
(-)\otimes_S \varphi^* 1_S
:
\sT(S)
\rightleftarrows
\sT(S;\varphi^*1)
:
U,
\]
where $U$ is the forgetful functor.
Note that $f^*$ in \eqref{coh.6.1} admits a right adjoint $f_*$ for all morphisms $f$ in $\lSch$ such that $f_*$ commutes with $U$.
Furthermore, if $f$ is log smooth,
then $f^*$ in \eqref{coh.6.1} commutes with $U$ and admits a left adjoint $f_\sharp$.

Since $\varphi^*$ is fully faithful,
we have $\varphi_\sharp \varphi^* 1_S\simeq 1_S$,
so we have an induced pair of adjoint functors
\[
\varphi_\sharp 
:
\sT(S;\varphi^* 1)
\rightleftarrows
\sT'(S)
:
\varphi^*.
\]
As before, $f_*$ commutes with this $\varphi^*$ for all morphisms in $\lSch$,
and $f^*$ commutes with this $\varphi^*$ for all log smooth morphisms in $\lSch$.
Observe that the two triangles
\begin{equation}
\label{coh.1.1}
\begin{tikzcd}
\sT(S;\varphi^*1)\ar[d,"U"']\ar[r,"\varphi_\sharp"]&
\sT'(S)&
\sT'(S)\ar[r,"\varphi^*"]\ar[rd,"\varphi^*"']&
\sT(S;\varphi^*1)\ar[d,"U"]
\\
\sT(S),\ar[ru,"\varphi_\sharp"']&
&
&
\sT(S)
\end{tikzcd}
\end{equation}
commutes.
\end{const}

\begin{prop}
\label{coh.7}
Let $S\in \lSch$.
Then for every strict projective smooth fs log scheme $X$ over $S$,
we have a natural isomorphism of $\varphi^* \unit_S$-modules
\[
M_S(X)\otimes_S \varphi^* 1_S
\simeq
\varphi^* M_S(X).
\]
\end{prop}
\begin{proof}
The morphism is obtained by adjunction from
\[
\varphi_\sharp(M_S(X)\otimes_S \varphi^*1_S)
\simeq
\varphi_\sharp M_S(X).
\]
Let $f\colon X\to S$ be the structural morphism.
Then by Theorem \ref{local.2},
$f_\sharp$ commutes with $\varphi^*$ since $f_*$ commutes with $\varphi^*$.
Hence
\[
M_S(X)\otimes_S \varphi^*1_S
\simeq
f_\sharp f^* \varphi^* 1_S
\simeq
\varphi^* f_\sharp f^*1_S
\simeq
\varphi^* M_S(X),
\]
where the first morphism is due to the $\lSm^\sat$ projection formula.
\end{proof}

\begin{prop}
\label{coh.1}
Let $S\in \Sch$.
Then for every projective $X\in \lSm_S$,
we have a natural isomorphism of $\varphi^* \unit_S$-modules
\[
M_S(X)
\otimes_S 
\varphi^* \unit_S
\simeq
\varphi^* M_S(X-\partial X).
\]
\end{prop}
\begin{proof}
The question is dividing Nisnevich local on $X$,
so we may assume $X\in \SmlSm_S$.
We proceed by induction on the number $d$ of components of $\partial X$.

The case of $d=0$ is a special case of Proposition \ref{coh.7}.
Assume $d>0$ and the claim holds for $d-1$.
Consider $Y\in \SmlSm_S$ and a divisor $D$ on $Y$ such that $(Y,D)\simeq X$.
We regard $D$ as an object of $\SmlSm_S$.
We have a morphism of cofiber sequences
\[
\begin{tikzcd}
M_S(X)\otimes_S \varphi^* 1_S\ar[r]\ar[d]&
M_S(Y)\otimes_S \varphi^* 1_S\ar[r]\ar[d]&
\Th_D(\rN_D Y)\otimes_S \varphi^* 1_S\ar[d]
\\
\varphi^* M_S(X-\partial X)\ar[r]&
\varphi^* M_S(Y-\partial Y)\ar[r]&
\varphi^* \Th_{D-\partial D}(\rN_{D-\partial D}(Y-\partial Y)).
\end{tikzcd}
\]
By induction,
the left and right vertical morphisms are isomorphisms,
so the middle vertical morphism is an isomorphism too.
\end{proof}

\begin{exm}
\label{coh.2}

For $S\in \lSch$,
consider the $\A^1$-localization functor
\[
\omega_\sharp \colon \logSH(S)
\to
\logSH(S)[(\A^1)^{-1}].
\]
Recall from \cite[Proposition 2.5.7]{logA1} that the right-hand side is equivalent to $\SH(S)$ if $S\in \Sch$.

Let us collect some examples of $\omega^*1_S$-algebras.
Let $S\in \Sch_S$.
For every $A\in \CAlg(\SH(S))$,
$\omega^*A$ is an $\omega^*1_S$-algebra.
In particular,
$\omega^*\KGL$ is an $\omega^*1_S$-algebra.

See \cite[Definition 8.5.3]{logSH} for $\blogTC$, $\blogTC^-$, and $\blogTHH$.
Together with the log cyclotomic trace $\omega^*\KGL\simeq \logKGL \to \blogTC$ in \cite[Theorem F]{logSHF1},
we see that $\blogTC$ is an $\omega^*1_X$-algebra.
It follows that $\blogTC^-$ and $\blogTHH$ are $\omega^*1_X$-algebras too.

A general recipe in \cite[Construction 6.5]{AHI} shows that
\begin{gather*}
\MZ^\syn
:=
(\gr_0^\BMS \logTC_p^\wedge,\gr_1^\BMS \logTC_p^\wedge,\cdots)\in \logSH(S)
\\
\text{(resp.\ }\bPrism
:=
(\gr_0^\BMS (\logTC^-)_p^\wedge,\gr_1^\BMS (\logTC^-)_p^\wedge,\ldots)\in \logSH(S))
\end{gather*}
in \cite[Definition 5.12]{BPO3} is a $\blogTC$-algebra (resp.\ $\blogTC^-$-algebra).
In particular, $\MZ^\syn$ and $\bPrism$ are $\omega^*1_S$-algebras.
See \cite{BLPO2} for the BMS filtrations on $\logTC_p^\wedge$ and $(\logTC^-)_p^\wedge$,
which generalizes the BMS filtration \cite{BMS19} to the log setting.
\end{exm}

\begin{rmk}
For a perfect field $k$ admitting resolution of singularities,
$\omega^1_k\simeq 1_k$ by \cite[Theorem 1.1(2)]{loghomotopy}.
In particular, every object of $\logSH(k)$ is an $\omega^*1_k$-module.
\end{rmk}

\begin{const}
\label{coh.3}
Let $S\in \Sch$.
For $\bE\in \logSH(S;\omega^*1)$ and $Y\in \Sch_S$,
consider
\[
\bE^\lc(Y/S)
:=
\hom_{\logSH(S;\omega^*\unit)}(\omega^* \Sigma^\infty g_!g^!\unit_Y,\bE).
\]
Here, $g\colon Y\to S$ is the structural morphism,
and $\lc$ stands for log compactification.
By \cite[Proposition 3.3.10]{CD19},
$\bE^\lc$ is a cdh sheaf.

For every projective $X\in \lSm_S$,
Proposition \ref{coh.1} yields
\[
\bE^\lc(X-\partial X/S)
\simeq
\hom_{\logSH(S)}(\Sigma^\infty X_+,\bE).
\]
\end{const}

\begin{thm}
\label{coh.8}
Let $k$ be a perfect field of characteristic $p>0$.
Then there exists a cdh sheaf of complexes
\(
R\Gamma_\crys^\lc
\)
on the category of schemes over $k$ such that
\[
R\Gamma_\crys^\lc(X-\partial X)
\simeq
R\Gamma_\crys(X)
\]
for every projective log smooth fs log scheme $X$ over $k$.
\end{thm}
\begin{proof}
By \cite[Theorem 4.30]{BLMP},
$\bPrism$ represents the crystalline cohomology.
Since $\bPrism$ is an $\omega^* 1_k$-module by Example \ref{coh.2},
we conclude by Construction \ref{coh.3}.
\end{proof}

This recovers a result of Annala-Hoyois-Iwasa \cite[Proposition 6.27]{AHI}.
Before that,
assuming resolution of singularities and weak factorization (resp.\ resolution of singularities),
Ertl-Shiho-Sprang \cite{ESS} (resp.\ Merici \cite{Mer25}) constructed such a cohomology theory,
and Merici represented it in $\DM(k)$.
Also, \cite[Theorem 1.1(2)]{loghomotopy} yields such a cohomology under resolution of singularities by considering
\[
\hom_{\logSH(S)}(\omega^* \Sigma^\infty X_+,\bE).
\]

\section{Failure of the localization property}
\label{local2}

We keep working with $\sT$ in \S \ref{sat}.
It is not the case that the localization property for $\logSH$ always holds.
If it always holds,
then the functor $i^*$ is an equivalence for every nil-immersion $i$ since the open complement of $i$ is empty.
However,
there are lots of non $\A^1$-invariant cohomology theories representable in $\logSH$, e.g.\ Hodge cohomology,
which are not nil-invariant.
In this section,
we show that the localization property also fails for closed non-open immersions of smooth schemes.

Throughout this section,
let $S\in \Sch$,
let $i\colon Z\to X$ be a closed immersion in $\Sm_S$,
let $j$ be the open complement of $i$,
let $u\colon (X,Z)\to X$ be the canonical morphism,
and let $f\colon X\to S$ be the structural morphism.

\begin{prop}
\label{local2.1}
There exists a natural fiber sequence
\[
(fu)_\sharp u^*\xrightarrow{\ad'} f_\sharp \xrightarrow{\ad} f_\sharp i_*i^*.
\]
\end{prop}
\begin{proof}
We have a commutative diagram
\[
\begin{tikzcd}
p_\sharp \ar[d,"\Gys"']\ar[r,"\ad"]&
p_\sharp i_*i^*\ar[d,"\Gys"]
\\
(pi)_\sharp \Sigma^{\rN_i}i^*\ar[r,"\ad"]\ar[rr,bend right=15,"\id"]&
(pi)_\sharp \Sigma^{\rN_i}i^*i_*i^*\ar[r,"\ad'"]&
(pi)_\sharp \Sigma^{\rN_i}i^*
\end{tikzcd}
\]
and the log Gysin sequence
\[
(pu)_\sharp u^* \xrightarrow{\ad}
p_\sharp
\xrightarrow{\Gys}
(pi)_\sharp \Sigma^{\rN_i} i^*.
\]
Hence it suffices to show that the composite natural transformation
\[
p_\sharp i_*i^*
\xrightarrow{\Gys}
(pi)_\sharp \Sigma^{\rN_i}i^*i_*i^*
\xrightarrow{\ad'}
(pi)_\sharp \Sigma^{\rN_i}i^*
\]
is an isomorphism.
This follows from Theorem \ref{local.1}.
\end{proof}

\begin{prop}
\label{local2.2}
Assume that $i$ is not an open immersion.
Then the natural null sequence
\[
j_\sharp j^* \xrightarrow{\ad'} \id \xrightarrow{\ad} i_*i^*
\colon
\logSH(X)\to \logSH(X)
\]
is not a fiber sequence.
\end{prop}
\begin{proof}
By Proposition \ref{local2.1},
it suffices to show that the natural transformation
\begin{equation}
\label{local2.2.1}
f_\sharp j_\sharp j^* \xrightarrow{\ad'} (fu)_\sharp u^*
\end{equation}
is not an isomorphism.
We can work Zariski locally on $X$.
We can also replace $Z\to X$ by $E\to \Bl_Z X$ with the exceptional divisor $E$.
Hence we may assume that $S=\Spec(A)$ and $Z=\Spec(A/a)$ for some ring $A$ and an element $a$ of $A$ such that $A\neq A_a$.

In this case,
assume that \eqref{local2.2.1} is an isomorphism.
Consider $\bOmega^0\in \logSH(S)$ in \cite[(7.7.4)]{logSH}.
Then we have
\[
\Hom_{\logSH(S)}((fu)_\sharp u^* \unit_S,\bOmega^0)
\simeq
\Hom_{\logSH(S)}(f_\sharp j_\sharp j^* \unit_S,\bOmega^0),
\]
which implies
\[
H^0((X,Z),\cO)
\simeq
H^0(X-Z,\cO).
\]
This is false since the left-hand (resp.\ right-hand) side is isomorphic to $A$ (resp.\ $A_f$).
\end{proof}

\bibliography{bib}
\bibliographystyle{siam}

\end{document}